\documentclass[pdflatex,sn-mathphys]{sn-jnl}
\usepackage[all,cmtip]{xy}

\jyear{2021}%

\theoremstyle{plain} 
\newtheorem{thm}{Theorem}[section] 
\newtheorem{cor}[thm]{Corollary} 
\newtheorem{lem}[thm]{Lemma} 
\newtheorem{prop}[thm]{Proposition}

\theoremstyle{definition} 
\newtheorem{defn}{Definition}[section] 
\newtheorem{exem}{Example} [section]

\theoremstyle{remark} 
\newtheorem{rem}{Remark}

\begin{document}

\title[Pullback functors for reduced and unreduced $L^{q,p}$-cohomology]{Pullback functors for reduced and unreduced $L^{q,p}$-cohomology}


\author*[1,2]{\pfx{Dr} \fnm{Stefano} \sur{Spessato} \sfx{ORCID: 0000-0003-0069-5853}\email{stefano.spessato@unicusano.it }}

\affil*[1]{\orgname{Università degli Studi "Niccolò Cusano"}, \orgaddress{\street{via don Carlo Gnocchi 3}, \city{Rome}, \postcode{00166}, \country{Italy}}}


\abstract{In this paper we study the reduced and unreduced $L^{q,p}$-cohomology groups of oriented manifolds of bounded geometry and their behaviour under uniform maps. A \textit{uniform map} is a uniformly continuous map such that the diameter of the preimage of a subset is bounded in terms of the diameter of the subset itself. In general, for each $p,q \in [1, +\infty)$, the pullback map along a uniform map does not induce a morphism between the between the spaces of $p$-integrable forms or even in $L^{q,p}$-cohomology. Then our goal is to introduce, for each $p$ in $[1, +\infty)$ and for each uniform map $f$ between manifolds of bounded geometry, an $\mathcal{L}^p$-bounded operator $T_f$, such that it does induce in a functorial way the appropriate morphism in reduced and unreduced $L^{q,p}$-cohomology.}




\maketitle
\subsection*{Acknowledgments}
I am grateful to Paolo Piazza, my advisor, and to Vito Felice Zenobi for the several discussions that we had and for the competences they shared with me. I would also like to thank Francesco Bei and Thomas Schick for their advice. Finally my thanks go to the reviewers for their patience in correcting my manuscript.
\\Data sharing not applicable to this article as no datasets were generated or analysed during the current study.
\section*{Introduction}
In this paper our goal is to define, for each $p \in [1, +\infty)$  and for each uniform map $f$ between manifolds of bounded geometry, a \textit{pullback operator} $T_f$ between the spaces of $p$-integrable forms. It is a well-known fact, indeed, that the classical pullback of an $\mathcal{L}^p$-form is not, in general, an $\mathcal{L}^p$-form. This means that the pullback is not inducing a well-defined morphism between the reduced and unreduced $L^{q,p}$-cohomology. On the other hand we prove that $T_f$ does induce some morphisms in (un)-reduced $L^{q,p}$-cohomology and we also prove that the functorial properties hold.
\\More formally, fix the category $\mathcal{C}$ which has oriented manifolds of bounded geometry as objects and uniform maps as arrows. Let  $\textbf{Vec}$ be the category which has complex vector spaces as objects and linear maps as arrows. Then we prove that, for every $z$ in $\mathbb{N}$,
	\begin{equation}
	\begin{cases}
	\mathcal{F}_z(M,g) = H^z_{q,p}(M) \\
	\mathcal{F}_z((M,g) \xrightarrow{f} (N,h)) = H^z_{q,p}(N) \xrightarrow{T_f} H^z_{q,p}(M)
	\end{cases}
	\end{equation}
	is a controvariant functor. The same happens if we define $\overline{\mathcal{F}}_z(M,g) := \overline{H}^z_{q,p}(M)$.
	Moreover, we will show that if two maps $f_1$ and $f_2$ are uniformly homotopic then $T_{f_1} = T_{f_2}$ in (un)-reduced $L^{q,p}$-cohomology and in $L^{q,p}$-quotient cohomology. Finally we will show that if the naive pullback $f^*$ does induce a morphism in reduced and unreduced $L^{q,p}$-cohomology, then $f^* = T_f$.
\\As a consequence of the existence of these functors, we obtain that the reduced and unreduced $L^{q,p}$-cohomology of a manifold of bounded geometry is invariant under uniform homotopy equivalence. This result is stated in Corollary \ref{finale}. Finally, as a consequence of the existence of the functors $\overline{\mathcal{F}}_z$, the invariance of the $L^2$-index of the signature operator defined by Bei at page 20 of \cite{Bei3} under uniform homotopy equivalence for manifolds of bounded geometry is proved.
\\The idea of the operator $T_f$ comes from the work of Hilsum and Skandalis \cite{hils}: in their paper, the authors define an $\mathcal{L}^2$-bounded operator $T_f$ for compact manifolds. Our operator $T_f$ is a \textit{bounded geometry version} of their one.
\\
\\The structure of the paper is the following: in the first section we introduce the objects and the arrows of the category $\mathcal{C}$. In particular, we show that every uniform map can be approximated by a smooth map with uniformly bounded derivatives. In the second one we introduce the reduced and unreduced $L^{q,p}$-cohomology. Moreover we also introduce  the \textit{Fiber Volume} of a Lipschitz submersion $\pi:(M,g) \longrightarrow (N,h)$. This is a real function defined on $(N,h)$ such that its boundedness implies the $\mathcal{L}^p$-boundedness of $\pi^*$.
In the third Section, given a smooth uniform map $f:(M,g) \longrightarrow (N,h)$, we introduce a Lipschitz submersion $p_f: f^*TN \longrightarrow N$ such that $p_f(0_x) = f(x)$. In the last section, we introduce a specific Thom form $\omega$ for the bundle $f^*TN$, we define $T_f$ and, in conclusion, we prove the functorial properties. Finally we prove, as a consequence of the main result, the uniform homotopy invariance of the $L^2$-index of the signature operator defined by Bei in \cite{Bei3}.
\section{Maps between manifolds of bounded geometry}
\subsection{Uniform maps and uniform homotopy}
Let us consider two metric spaces $(X,d_X)$ and $(Y, d_Y)$.
\begin{defn}
A map $f:(X,d_X) \longrightarrow (Y,d_Y)$ is \textbf{uniformly continuous} if for each $\epsilon > 0$ there is a $\delta(\epsilon) >0$ such that for each $x_1$, $x_2$ in $X$
\begin{equation}
d_X(x_1, x_2) \leq \delta(\epsilon) \implies d_Y(f(x_1), f(x_2)) \leq \epsilon.
\end{equation}
Moreover $f$ is \textbf{uniformly (metrically) proper} if for each $R\geq 0$ there is a number $S(R) >0$ such that for each subset $A$ of $(Y, d_Y)$
	\begin{equation}
	diam(A) \leq R \implies diam(f^{-1}(A)) \leq S(R).
	\end{equation}
A map $f:(X, d_X) \longrightarrow (Y,d_Y)$ is a \textbf{uniform map} if it is uniformly continuous and uniformly proper.
\end{defn}
\begin{rem}
Compositions of uniform maps are also uniform.
\end{rem}
Let $(X, d_X)$ and $(Y, d_Y)$ be two metric spaces and fix two actions of a group $\Gamma$ on $X$ and $Y$. Assume that $\Gamma$ acts by isometries on $X$ and on $Y$.
\begin{defn}
	Two maps $f_0$ and $f_1: (X, d_X) \longrightarrow (Y,d_Y)$ are \textbf{$\Gamma$-uniformly-homotopic} if they are $\Gamma$-equivariant maps and they are homotopic with a uniformly continuous homotopy $H: (X\times[0,1], d_X \times d_{[0,1]}) \longrightarrow (Y, d_Y)$ which is $\Gamma$-equivariant\footnote{The action of $\Gamma$ on $X \times [0,1]$ is defined as $\gamma(x,t) := (\gamma x, t)$ for each $\gamma$ in $\Gamma$}.
	\\We will denote it by
	\begin{equation}
	f_1 \sim_{\Gamma} f_2.
	\end{equation}
	Moreover $f_1$ and $f_2$ are \textbf{$\Gamma$-Lipschitz-homotopic} if $f_1 \sim_\Gamma f_2$ and $H$ is a Lipschitz map.
\end{defn}
\begin{rem}\label{uniformlyp}
	Consider two uniformly continuous maps $F,f: (X, d_X) \longrightarrow (Y,d_Y)$ such that $f \sim_\Gamma F$. Then, if $f$ is a uniformly proper map, also the homotopy $h$ is uniformly proper. As a consequence of this, in particular, also $F$ is uniformly proper.
	\\In order to prove it fix $\epsilon >0$. Since $h$ is uniformly continuous, there is a $\delta(\epsilon)$ such that
	\begin{equation}
	 d(t, s) \leq \delta(\epsilon) \implies d(h(x,t), h(x,s)) \leq \epsilon.
	\end{equation}
	Fix $t$ in $[0,1]$. If we divide $[0,t]$ in $N_t$ intervals of length less or equal to $\delta(\epsilon)$, we obtain that
	\begin{equation}
	\begin{split}
	d(f(x), h(x,t)) &\leq d(h(x,0),h(x,t)) \\ 
	&\leq d(h(x,0),h(x,t_1)) + ... + d(h(x, t_{N_s-1}),h(x,t))\\
	&\leq N_t \epsilon =: R(t).
	\end{split}
	\end{equation}
	Observe that $R(t_1) \leq R(t_2)$ if $t_1 \leq t_2$ and $R(1) = \frac{\epsilon}{\delta(\epsilon)}$.
	\\If $A$ is a subset of $Y$ and $h_t: X \longrightarrow Y$ is defined as $h_t(p) :=h(p,t)$, then
	\begin{equation}
	\begin{split}
	h_t^{-1}(A) &= \{p \in X\vert h(p,t) \in A\} \\
	&\subseteq \{p \in X\vert f(p) \in  B_{R(t)}(A)\} = f^{-1}(B_{R(t)}(A)) \label{eqno}
	\end{split}
	\end{equation}
	where $B_{R(t)}(A)$ are the points $y$ of $Y$ such that $d(y, A) \leq R(t)$. Then 
	\begin{equation}
	h^{-1}(A) \subseteq f^{-1}(B_{R(1)}(A)) \times [0,1].
	\end{equation}
	and consequently, since $f$ is uniformly proper, and since $diam(B_{R(1)}(A))\leq diam(A) + 2 R(1)$,
	\begin{equation}
	\begin{split}
	diam(h^{-1}(A)) &\leq diam(f^{-1}(B_{R(1)}(A))) + diam([0,1]) \\
	&\leq S(diam(A) + 2 R(1)) + 1.
	\end{split}
	\end{equation}
\end{rem}
\begin{defn}
	A map $f:(X, d_X) \longrightarrow (Y,d_Y)$ is a \textbf{$\Gamma$-uniform homotopy equivalence} if $f$ is $\Gamma$-equivariant, uniformly continuous and there is a $\Gamma$-equivariant map $g:(y,d_y) \longrightarrow (X, d_X)$ such that
	\begin{itemize}
		\item $g$ is a homotopy inverse of $f$,
		\item $g$ is uniformly continuous,
		\item $f\circ g \sim_\Gamma id_N$ and $g \circ f \sim_\Gamma id_M$.
	\end{itemize}
\end{defn}
\subsection{Manifolds of bounded geometry}
In this section we introduce the notion of manifolds of bounded geometry. All the definitions and propositions below can be found in Chapter 2 of the thesis of Eldering \cite{bound}.
\begin{defn}
	A Riemannian manifold $(M,g)$ has \textbf{$k$-bounded geometry} if:
	\begin{itemize}
		\item the sectional curvature $K$ of $(M,g)$ and its first $k$-covariant derivatives are bounded, i.e. $\forall i = 0,...,k$ there is a constant $V_i$ such that $\vert\nabla^i K(x) \vert \leq V_i$ for all $x$ in $M$.
		\item there is a number $C > 0$ such that for all $p$ in $M$ the injectivity radius $inj_M(p)$ in $p$ satisfies $inj_M(p) \geq C.$
		The maximal number which satisfies this inequality will be denoted by $inj_M$.
	\end{itemize}
	When we talk about a manifold $M$ with bounded geometry, without specifying $k$, we mean that $M$ has $k$-bounded geometry for all $k$ in $\mathbb{N}$.
\end{defn}
\begin{defn}
	Let $(M,g)$ be a manifold of bounded geometry. We define $\delta > 0$ to be \textbf{$M$-small} if, for each $k >0$ in $\mathbb{N}$,
	\begin{itemize}
	    \item there is a constant $C_k > 0$ such that for all $x_1,x_2 \in M$ with $d(x_1,x_2) < \delta$  the coordinate transition map $\phi_{2,1} = exp^{-1}_{x_2} \circ exp_{x_1}: U \longrightarrow T_{x_2}M$ where $ U = exp_{x_1}^{-1}(B_{\delta}(x_1)\cap B_\delta(x_2)) \subset T_{x_1}M$ is $C^{k-1}$-bounded\footnote{This means that all the derivatives of degree less or equal to $k-1$ are bounded.} with $\vert\phi_{2,1}\vert_{k-1} \leq C_k$.
	\item the metric up to its $k$-th order derivatives and the Christoffel symbols up to its $(k-1)$-th order derivatives are bounded in normal coordinates of radius $\delta$ around each $x \in M$, with bounds that are uniform in $x$.
	\end{itemize}
\end{defn}
As reported in page 45 of \cite{bound}, every manifold $(M,g)$ of bounded geometry admits an $M$-small $\delta$.
\begin{rem}\label{bvolume}
		Consider a Riemannian manifold $(M,g)$ such that
		\begin{equation}\label{riccibound}
		Ric(M) \geq (n-1)C
		\end{equation}
		where $C$ is a constant. In particular, if $M$ has $k_{\geq 0}$-bounded geometry, then (\ref{riccibound}) is satisfied using $C = - V_0$ where $V_0$ bounds the norm of sectional curvature.  Let us denote the measure $\mu_M$ induced by $g$ on $M$. Because of the Bishop-Gromov inequality, if $(M,g)$ satisfies (\ref{riccibound}), then for each $p \in M$ and for each $r \geq 0$ we obtain $\mu_M(B_r(p)) \leq Q(r)$ for some function $Q: \mathbb{R}_{\geq 0} \longrightarrow  \mathbb{R}_{\geq 0}$. So, if $A \subset M$ has $diam(A) = r$, 
		\begin{equation}
		\mu_M(A) \leq \mu_M(B_{diam(A)}(p)) \leq C(diam(A)),
		\end{equation}
		where $p$ is a point of $A$.
	\end{rem}
	\subsection{Uniformly proper and discontinuous actions}
\begin{defn}
	Consider $\Gamma$ a group which acts by isometries on a metric space $(X, d_X)$. The action of $\Gamma$ is \textbf{free and uniformly properly discontinuous (FUPD)} if 
	\begin{itemize}
		\item  the action of $\Gamma$ is free,
		\item  there is a number $\delta > 0$ such that $d_X(x, \gamma x) \leq \delta \implies x = \gamma x.$
	\end{itemize}
\end{defn}
\begin{rem}
If the action of $\Gamma$ is FUPD, then it is properly discontinuous and free. 
\end{rem}
\begin{prop}\label{Gammaq}
	Let $(M,g)$ be a manifold of bounded geometry. Consider $\Gamma$ a group acting on $M$ by isometries. Suppose the action of $\Gamma$ is free and properly discontinuous. Then the following statements are equivalent:
	\begin{enumerate}
		\item a group $\Gamma$ acting FUPD on $M$ and $N$,
		\item the quotient $M/\Gamma$ has bounded geometry\footnote{We are considering on $M/\Gamma$ the Riemannian metric induced by $g$}.
	\end{enumerate}
	\end{prop}
	\begin{proof}
	$1 \implies 2$. Let $\delta$ be the constant of the FUPD action of $\Gamma$. Then $B_\delta(p)$ is a trivializing open of $M/ \Gamma$  for each $p$ in $M$. So $inj_{M/\Gamma} \geq \min\{inj_M, \delta\}$ and the curvature of $M/\Gamma$ has the same bounds of the curvature of $M$.
	\\$2 \implies 1$. Let us suppose that for each $\delta > 0$ there is a point $p$ and a $\gamma \in \Gamma$ such that $d_M(p, \gamma p) \leq \delta$. Suppose $\delta < \min\{inj_M, inj_{M/\Gamma}\}$. Then there is a vector $v$ in $T_pM$ whose norm equals $\delta$ and $exp_p(v) = \gamma p$. Consider the Riemannian covering $s: M \longrightarrow M/\Gamma$. Observe that
	\begin{equation}
	s \circ exp_p = exp_{s(p)} \circ ds
	\end{equation}
	Then
	\begin{equation}
	s(\gamma p) = s \circ exp_p (v) = exp_{s(p)} \circ ds(v).
	\end{equation}
	Since $ds$ is an isometry, the norm of $ds(v)$ is less or equal to $inj_{M/\Gamma}$. This means that $ds(v) = 0$ and, in particular, $v = 0$. Then
	\begin{equation}
	\gamma p = exp_p(v) = exp_p(0) = p.
	\end{equation}
	Because the action of $\Gamma$ is free, $\gamma = id_\Gamma$
\end{proof}
\subsection{Lipschitz approximation of a uniformly continuous map}\label{approx}
Let $(M,g)$ and $(N,h)$ be two manifolds of bounded geometry and $f \in C^k(M,N)$.
\begin{defn}
	 A map $f$ is of class $C^k_b$ if there exist $M$-small $\delta_M$ and an $N$-small $\delta_N > 0$ such that for each $x \in X$ we have $f(B_{\delta_M}(x)) \subset B_{\delta_N}(f (x))$ and the composition
	\begin{equation}
		F_p = exp^{-1}_{f(p)} \circ f \circ exp_p : B_{\delta_M}(0) \subset T_pM \longrightarrow T_{f(p)}N
	\end{equation}
	in normal coordinates is of class $C^k_b$ and its $C^l$-norms ($l= 0,...,k$) as function from $T_pM$ to $T_{f(p)}N$ are uniformly bounded in $x \in X$.
\end{defn}
\begin{rem}
	There are some remarks we wish to make:
	\begin{itemize}
		\item The assumption of bounded geometry is necessary in order to define $C^k_b$-maps. More details can be found in pages 44-45 of \cite{bound}.
		\item Composition of two $C^k_b$-maps is a $C^k_b$-map.
		\item Uniformly continuous maps are $C^0_{b}$-maps.
		\item A $C^1_b$-map is a Lipschitz map.
	\end{itemize} 
\end{rem}
Consider a uniform map $f: (M,g) \longrightarrow (N,h)$ between manifolds of bounded geometry. Let us suppose, moreover, that there is a group $\Gamma$ on $M$ and $N$ acting FUPD and assume that $f$ is $\Gamma$-equivariant. We want to show that there is a $\Gamma$-equivariant uniform map $F: (M,g) \longrightarrow (N,h)$ which is $C^k_{b}$ for all $k$ and that $F \sim_\Gamma f$.
\\In order to prove this fact we need the Lemma 2.34 of the thesis of Eldering \cite{bound}.
\begin{lem}\label{primaappr}
	Let $r, h > 0$, $g \in C^{k}(B_{r + 2 h}(0) \subset \mathbb{R}^m, \mathbb{R}^n)$. Then for each $\epsilon > 0$ there is a $\nu_0$ such that $g$ can be approximated by a function $G_\nu$ for each $0 < \nu \leq \nu_0$ such that
	\begin{itemize}
		\item $G_\nu = g$ outside $B_{r + h}(0)$;
		\item $G_\nu$ is a smooth function in $B_{r}(0)$;
		\item $\vert g - G_\nu\vert_{s} \leq \epsilon$ for each $s \leq k$  ;
		\item $\vert G_\nu\vert_{l} \leq C(\nu, l)\vert g\vert_0$ on $B_r(0)$  for some $C_{\nu,l} >0$.
	\end{itemize}
	Note that $C(\nu,l)$ may grow unboundedly as $\nu \longrightarrow 0$ or $l \longrightarrow +\infty$.
\end{lem}
\begin{rem}\label{quindi}
	The approximation above is defined in the following way.  Consider a mollifier, i.e. a compactly supported, smooth, positive function $\phi: \mathbb{R}^n \longrightarrow \mathbb{R}$ such that its integral over $\mathbb{R}^n$ equals to 1. Suppose that the support of  $\phi$ is contained in the Euclidean ball of radius $1$ in $\mathbb{R}^m$. Let us consider a non-increasing smooth map $\chi: [0, r + 2 h] \longrightarrow [0,1]$ which is $1$ if $x \leq r$ and $0$ if $x\geq r+ h$. Then we define
	\begin{equation}
			G_\nu(x) = (1 - \chi(\vert\vert x\vert\vert))g(x) + \chi(\vert\vert x\vert\vert)\int_{B^m_1(0)}g(x - \nu y)\phi(y) dy.
	\end{equation}
	Observe that if $\nu = 0$ then $G_\nu = g$.
\end{rem}
\begin{rem}\label{epsnu}
	Fix $\epsilon > 0$. By the proof of Lemma 2.34 of \cite{bound}, if $\nu \leq \delta(\epsilon)$ where $\delta(\epsilon)$ satisfies
	\begin{equation}\label{qui}
	    d(p,q) \leq \delta(\epsilon) \implies d(g(p), g(q)) \leq \epsilon
	\end{equation}
	then
	\begin{equation}
		d(g(p), G_\nu(p)) \leq \epsilon.
	\end{equation}
	Let us define $\delta_\chi: \mathbb{R}_{>0} \longrightarrow \mathbb{R}_{>0}$ a function such that for each $\epsilon >0$
	\begin{equation}
	\vert x_1 - x_0 \vert \leq \delta_{\chi}(\epsilon) \implies \vert \chi(\vert\vert x_1\vert\vert) - \chi(\vert\vert x_0\vert\vert) \vert \leq \epsilon.
	\end{equation}
	Observe that, without loss of generality, we can always suppose that $\delta: \mathbb{R}_{>0} \longrightarrow \mathbb{R}_{>0}$ is non-decreasing and, at the same time,
	\begin{equation}
	    \delta(\epsilon) \leq \delta_{\chi}(\frac{\epsilon}{\vert g \vert_0})
	\end{equation}
	where $\vert g \vert_0 := \max\limits_{p \in B_1^m(0)}\{\vert g(p) \vert \}$. Under this assumption on $\delta$, if we define $\tilde{\delta}(\epsilon) := \delta(\frac{\epsilon}{4})$, we obtain that
	\begin{equation}
	  d(p,q) \leq \tilde{\delta}(\epsilon) \implies d(G_\nu(p), G_\nu(q)) \leq \epsilon.
	\end{equation}
\end{rem}
We are ready to prove the existence of a $C^k_{b}$-approximation of a uniformly continuous map.
\begin{prop}\label{appr}
	Consider two Riemannian manifolds $(M,g)$ and $(N,h)$ of bounded geometry and let $f:(M,g) \longrightarrow (N,h)$ be a uniformly continuous map. Fix $\epsilon >0$ small enough. Then there is map $F:(M,g) \longrightarrow (N,h)$ such that
	\begin{itemize}
		\item $d(F(p), f(p)) \leq \epsilon$ for all $p$ in $M$,	
		\item for all $l \geq 1$  the approximation $F$ is a smooth $C^l_{b}$-map,
		\item Consider $\Gamma$ a group that acts FUPD by isometries on $M$ and $N$. Assume that $f$ is $\Gamma$-equivariant. Then $F$ can be chosen to be $\Gamma$-equivariant,
		\item $f \sim_\Gamma F$.
	\end{itemize}
This implies that if $f$ is a uniform map, then, for each $\epsilon > 0$, there is a smooth, uniformly proper and Lipschitz approximation $f_\epsilon$ of $f$.
\end{prop}
\begin{proof}
	Since $f$ is uniformly continuous, there are two positive numbers $\sigma_1 < \sigma_2$ both $N$-small and there is an $M$-small number $\delta_f$ such that $f(B_{\delta_f}(q)) \subseteq B_{\sigma_1}(f(q))$.\\
	Let us consider the Riemannian manifold $X := M/\Gamma$. Let us define $R_0 := \frac{1}{4} min\{\delta_f, inj_X\}$.
	\\Following Proposition \ref{Gammaq}, observe that $R_0 \leq inj_M$ and $R_0 \leq \delta_0$ where $\delta_0$ is the constant given by the FUPD action of $\Gamma$. Fix then $\delta_1 < \delta_2 < R_0$ where $\delta_2$ is $X$-small. 
	\\By Lemma 2.16 of \cite{bound}, there is a number $K$ and a countable cover of $X$ given by $\{B_{\delta_1}(x_i)\}$ such that for all $x$ in $X$ the ball $B_{\delta_2}(x)$ intersects at most $K$ balls of $\{B_{\delta_2}(x_i)\}_i$.
	\\Consider the preimage of the cover $\{B_{\delta_2}(x_i)\}_i$ on $M$. Since $\delta_2 < \delta_0$, it has the form $ \{ \sqcup_{\gamma \in \Gamma} B_{\delta_2}(\gamma \tilde{x}_i)\}_i$, where $\tilde{x}_i$ is an element of the fiber of $x_i$.
	\\Moreover $f(B_{\delta_2}(q)) \subseteq B_{\sigma_1}(f(q))$ for all $q$ in $M$.
	\\
	\\Fix $\epsilon < \sigma_2 - \sigma_1$ and let $C$ be a number which is greater then the Lipschitz constants of $\exp_{f(\gamma \tilde{x}_i)}$, $ exp_{\gamma \tilde{x}_i}$ and their inverses. Fix $F_0 := f$. Then define for all $i$ in $\mathbb{N}$
	\begin{equation}
		F_{i+1}(p) := \begin{cases} \exp_{f(\gamma \tilde{x}_i)} \circ G_{i,\gamma, \nu} \circ exp_{\gamma \tilde{x}_i}^{-1} \mbox{  if   } p \in B_{R_0}(\bigcup_\gamma \{\gamma \tilde{x}_i\}) \\
			F_i(p) \mbox{   otherwise} \end{cases}
	\end{equation}
	where  $G_{i,\gamma,\nu}: B_{\delta_2}(0)\subset T_{\gamma \tilde{x}_i} M \cong \mathbb{R}^m \longrightarrow B_{\sigma_2}(0)\subset T_{f(\gamma \tilde{x}_i)} N \subset \mathbb{R}^n$ is a $C^k_{b}$-approximation of $g_{i,\gamma} = \exp_{f(\gamma \tilde{x}_i)}^{-1} \circ F_i \circ exp_{\gamma \tilde{x}_i}$ defined by using a function $\chi$ as in Remark \ref{quindi} and $\nu$ such that
	\begin{equation}
	   d(p, q) \leq \nu \implies d(f(p), f(q)) \leq \frac{\epsilon}{CK}.
	\end{equation}
	In particular, in order to obtain the $\Gamma$-equivariance of $F_{i+1}$, we choose for each $i$ a $ \tilde{x}_i$ in the fiber of $x_i$. Fix an orthonormal basis of $T_{\tilde{x}_i}M$ and an orthonormal basis of $T_{f(\tilde{x}_i)}N$. We define $G_{i, e, \nu}$ with respect to the coordinates on the tangent spaces induced by these basis. Then we define for each $\gamma$ in $\Gamma$
	\begin{equation}
	 G_{i,\gamma,\nu} := d\gamma \circ G_{i,e,\nu} \circ d \gamma^{-1}.
	\end{equation}
	Observe that, by the Gluing lemma (Theorem III.9.4, p. 83. of the book of Dugundji \cite{Dug}), $F_i$ is well-defined. Finally we define
	\begin{equation}
		F(p) := \lim_{i \rightarrow + \infty} F_{i}(p).
	\end{equation}
	Let us check all the properties of $F$:
	\begin{itemize}
		\item $F$ is well-defined: consider $p$ a point in $M$. Let $s$ be the covering $s: M \longrightarrow X$. Then $s(B_{\delta_2}(p))$ intersects at most $K$ balls $B_{\delta_2}(x_i)$. Moreover, since $\delta_2 < \frac{1}{4}\delta_0$, then for all of these $x_i$ there exists only one $\gamma$ in $\Gamma$ such that $B_{\delta_2}(p)$ intersects $B_{\delta_2}(\gamma \tilde{x}_i)$.
		\\For all $p$ there are at most $K$ indexes $i_j$ such that $F_{i_j}(p) \neq F_{i_j + 1}(p)$ and so the limit exists since the sequence $\{F_i(p)\}_{i \geq i_K + 1}$ is constant.
		\item Let us denote by $x := exp^{-1}_{\gamma \tilde{x}_i}(p)$. We obtain $d(F(p), f(p)) \leq \epsilon.$
		Indeed, by applying Proposition \ref{primaappr} and  Remark \ref{epsnu}, 
		\begin{equation}
		\begin{split}
			d(f(p), F(p)) &\leq d(f(p), F_{i_1}(p)) + ... + d(F_{i_K}(p), F_{i_{K+1}}(p))\\
			&\leq C (\vert g_{0,\gamma, \nu}(x) - G_{0, \gamma, \nu}(x)\vert + ... + \vert g_{i_K, \gamma, \nu}(x) - G_{i_K, \gamma, \nu}(x) \vert )\\ 
			&\leq C \cdot K\cdot \frac{\epsilon}{C K} = \epsilon
			\end{split}
		\end{equation}
		\item $F$ is a $C^k_{b}$-map for each $k$ in $\mathbb{N}$. Consider a point $p$. As a consequence of Lemma 2.16 of \cite{bound}, there is a finite sequence $i_1,..., i_K$ of the indexes $i$ such that $F_{i_j+1} \neq F_{i_j}$ on $B_{\delta_2}(p)$. So, we have to study the boundedness of the derivatives of
		\begin{equation}
			\begin{split}
				F_p :&= exp^{-1}_{F_{i_K +1}(p)} \circ F_{i_K + 1} \circ exp_{p}\\
				&= exp^{-1}_{F_{i_K +1}(p)} \circ \exp_{f(\gamma x_K)} \circ G_{i_K,\gamma, \nu} \circ exp_{\gamma x_{i_K}}^{-1} \circ exp_{p}.
			\end{split}
		\end{equation}
		Observe that $exp^{-1}_{F_{i_K +1}(p)} \circ \exp_{f(\gamma \tilde{x}_i)}$ and $exp_{\gamma x_{i_K}}^{-1} \circ exp_{p}$ are changes of normal coordinates, and so, since we are in a bounded geometry setting, they are $C^k_{b}$-maps with uniformly bounded norm. Consider $G_{i_K,\gamma, \nu}$. By applying Proposition \ref{primaappr}, we obtain
		\begin{equation}\label{haha}
			\vert G_{i_K,\gamma, \nu}\vert_{k} \leq C(\nu , k)\vert g_{i_{K},\gamma, \nu}\vert_0 \leq C(\nu, k)\sigma_2.
		\end{equation}
		\item $F$ is $\Gamma$-equivariant. In order to prove this we have to check that all the $F_i$ are $\Gamma$-equivariant. Observe that $F_0 = f$ is $\Gamma$-equivariant. Moreover, by the definition of $F_i$, if $F_i$ is $\Gamma$-equivariant, then also $F_{i+1}$ is $\Gamma$-equivariant. We conclude by observing that
		\begin{equation}
			\gamma F(p) = \gamma F_{i_{K}+1}(p) = F_{i_{K}+1} (\gamma p) = F(\gamma p).
		\end{equation}
	\end{itemize}
	In order to conclude the proof we have to show that $F$ and $f$ are $\Gamma$-uniformly-homotopic. \\Let us define the map $H_0: (M \times [0,1], g + dt^2) \longrightarrow (N,h)$ as $H_0(p,t) := f(p)$. Fix for each $i$ a $\tilde{x}_i$ in the fiber of $x_i$. Fix an orthonormal basis of $T_{\tilde{x}_i}M$ and an orthonormal basis of $T_{f(\tilde{x}_i)}N$. We denote by $H_{i, e, \nu}$ the $C^k_{b}$-approximation\footnote{With respect to the coordinates on the tangent spaces induced by these basis. } of 
	\begin{equation}
	h_{i,\gamma,t} = \exp_{f(\gamma \tilde{x}_i)}^{-1} \circ H_{i}(\cdot, t) \circ exp_{\gamma \tilde{x}_i}
	\end{equation}
	defined by using $t \cdot \nu$ and the function $\chi$. \\Then, for each $\gamma$ in $\Gamma$ we have $H_{i,\gamma,t}: B_{\delta_2}(0)\subset T_{\gamma \tilde{x}_i} M \cong \mathbb{R}^m \longrightarrow B_{\sigma_2}(0)\subset T_{f(\gamma \tilde{x}_i)} N \subset \mathbb{R}^n$ which is the map 
	\begin{equation}
	 H_{i,\gamma, t} := d\gamma \circ H_{i,e, t \cdot \nu} \circ d \gamma^{-1}.
	\end{equation}So, for each $i \in \mathbb{N}$, we define 
	\begin{equation}
		H_{i+1}(p,t) := \begin{cases} \exp_{f(\gamma \tilde{x}_i)} \circ H_{i,\gamma, t } \circ exp_{\gamma \tilde{x}_i}^{-1} \mbox{  if   } p \in B_{inj_M}(\cup_\gamma \{\gamma \tilde{x}_i\}) \\
			H_i(p, t) \mbox{   otherwise} \end{cases}
	\end{equation}
	Finally we define the map $H: (M \times [0,1], g + dt^2) \longrightarrow (N,h)$ as
	\begin{equation}
		H(p, t) = \lim_{i \rightarrow +\infty} H_{i}(p, t).
	\end{equation}
	As well as for $F$, also $H$ is well-defined and $\Gamma$-equivariant.
	\\We will check the uniformly continuity of $H$ in the following way. First we observe that, by definition, $H_0$ is uniformly continuous. Then for each $\varepsilon$ there is a $\delta_0(\epsilon)$ such that 
	\begin{equation}
	 d((p_0, t), (q, s)) \leq \delta_0(\varepsilon) \implies d(H_0(p_0,t), H_0(q,s)) \leq \varepsilon.
	\end{equation}
	Observe that, since $H_0$ is uniformly continuous, then $\delta_0$ does not depends on $(p_0,t)$. Moreover, without loss of generality we can also suppose that $\delta_0 (\varepsilon) \leq \delta_{\chi}(\frac{\varepsilon}{\sigma_2})$ and that $\delta_0$ is non-decreasing with respect to $\varepsilon$.
	\\Let us assume now that $H_i$ is continuous. Fix an $\varepsilon$ and a point $(p,t)$ and let us denote by $\delta_i(\varepsilon, p, t)$ a number such that, for each $q$ in $M$ 
	\begin{equation}
	d((p,t),(q,s)) \leq \delta_i(\varepsilon, p, t) \implies d(H_i(p,t), H_i(q,s)) \leq \varepsilon.
	\end{equation}
	Assume that $\delta_i(\varepsilon, p, t) \leq \delta_{\chi}(\frac{\varepsilon}{\sigma_2})$ and that $\delta_i(\cdot, p,t)$ is non-decreasing with respect to $\varepsilon$ for each $p$ and $t$.
	\\Choose $p_0$ in $M$: if $B_{\delta_2}(p_0) \cap [\sqcup_{\gamma \in \Gamma} B_{\delta_2}(\gamma \tilde{x}_i)] = \emptyset$, then
	\begin{equation}
	d((p_0,t),(q,s)) \leq \delta_i(\varepsilon, p_0, t) \implies d(H_{i + 1}(p_0,t), H_{i+1}(q,s)) \leq \varepsilon.
	\end{equation}
	So we can set
	\begin{equation}
	\delta_{i+1}(\varepsilon, p_0, t) := \delta_{i}(\varepsilon, p_0, t).
	\end{equation}
	Let us suppose that $ B_{\delta_2}(p_0) \cap [\sqcup_{\gamma \in \Gamma} B_{\delta_2}(\gamma \tilde{x}_i)] \neq \emptyset$.
	In this case we have to study the uniform continuity of
	\begin{equation}
	\exp_{f(\gamma \tilde{x}_i)} \circ H_{i,\gamma, t} \circ exp_{\gamma \tilde{x}_i}^{-1}.
	\end{equation}
	In particular, we start studying the distance
	\begin{equation}
	 d(H_{i+1}(p_0,t), H_{i+1}(p_0, s)).
	\end{equation}
	Recall that $h_{i,\gamma,t}: T_{\gamma \tilde{x}_i}M \longrightarrow T_{f(\gamma \tilde{x}_i)}N$ is defined as
	\begin{equation}
	h_{i,\gamma,t} = \exp_{f(\gamma \tilde{x}_i)}^{-1} \circ H_{i}(\cdot, t) \circ exp_{\gamma \tilde{x}_i}
	\end{equation}
	and that $H_{i,\gamma, t}$ is its approximation defined by using $t\cdot \nu$ and $\chi$.
	\\Observe that the exponential maps and their inverses are Lipschitz maps with uniformly bounded constant $C$. Assume $C >1$.
	\\Fix a $\varepsilon >0$ and let us denote by $x_0 := exp^{-1}_{\gamma \tilde{x}_i}(p_0)$. Define
	\begin{equation}
	\delta_{h,i}(\varepsilon, x_0, t) := \frac{1}{C}\delta_i(\frac{\varepsilon}{C}, p_0, t).
	\end{equation}
	Observe that $\delta_{h,i}(\varepsilon, x_0, t) \leq \delta_i(\varepsilon, p_0, t) \leq \delta_{\chi}(\frac{\varepsilon}{\sigma_2})$. Moreover, $h_{i, \gamma, t}$ is also continuous in $t$, indeed,
	\begin{equation}
	    \vert (x_0,t) - (x_2, s)\vert \leq \delta_{h,i}(\varepsilon, x_0, t) \implies d(h_{i,\gamma, t}(x_0), h_{i,\gamma, s}(x_2)) \leq \varepsilon.
	\end{equation}
	Then for any $s$ in $[0,1]$ such that $\vert s - t\vert \leq \delta_{h,i}(\frac{\varepsilon}{2}, x_0, t)$, then
		\begin{equation}
		\begin{split}
			&\vert H_{i,\gamma, s}(x_0) - H_{i,\gamma, t}(x_0)\vert \leq (1 - \chi(\vert \vert x_0 \vert \vert))\vert h_{i,\gamma,s}(x_0) - h_{i,\gamma,t}(x_0)\vert \\
			& + \chi(\vert \vert x_0 \vert \vert) \int_{B^m} \vert h_{i,\gamma,s}(x_0 - s\cdot \nu y) - h_{i,\gamma,s}(x_0 - t\cdot \nu y)\vert \phi(y) dy \\
			& + \chi(\vert \vert x_0 \vert \vert) \int_{B^m} \vert h_{i,\gamma,s}(x_0 - t\cdot \nu y) - h_{i,\gamma,t}(x_0 - t\cdot \nu y)\vert\phi(y) dy\\ &\leq (1 - \chi(\vert \vert x_0 \vert \vert))\frac{\varepsilon}{2} + \chi(\vert \vert x_0 \vert \vert)\frac{\varepsilon}{2} + \frac{\varepsilon}{2}  \leq \varepsilon.
		\end{split}
	\end{equation}
	This means that, given
	\begin{equation}
	\tilde{\delta_i}(\varepsilon, p_0, t) := \frac{1}{C}\delta_{h,i}(\frac{\varepsilon}{2C}, p_0, t) = \frac{1}{C^2}\delta_i(\frac{\varepsilon}{2C^2}, p_0, t),
	\end{equation}
	then, since $H_{i+1} = \exp_{f(\gamma \tilde{x}_i)} \circ H_{i,\gamma, t} \circ exp_{\gamma \tilde{x}_i}^{-1}$,
	\begin{equation}
	    d((p_0,t),(p_0,s)) \leq \tilde{\delta_i}(\varepsilon) \implies d(H_{i+1}(p_0,t), H_{i+1}(p_0,s)) \leq \varepsilon.
	\end{equation}
	Let us consider, now, the distance
	\begin{equation}
	d(H_{i+1}(p_0, t), H_{i+1}(q, t)).
	\end{equation}
	Let us denote by $\Delta_{i}(\varepsilon) := \delta_{h,i}(\frac{\varepsilon}{4})$. Then, by Remark \ref{epsnu} and by 
	\begin{equation}
	\delta_{h,i}(\frac{\varepsilon}{4}) \leq \delta_{i}(\frac{\varepsilon}{4}) \leq \delta_{\chi}(\frac{\varepsilon}{4\sigma_2}) \leq \delta_{\chi}(\frac{\varepsilon}{4\vert h_{i,\gamma,t}\vert_0})
	\end{equation}
	we obtain that
	\begin{equation}
	d(x_1, x_2) \leq \Delta_{i}(\varepsilon) \implies d(H_{i,\gamma, t }(x_1), H_{i,\gamma, t }(x_2)) \leq \varepsilon.
	\end{equation}
	Moreover, if we define
	\begin{equation}
	 \hat{\delta_i}(\varepsilon) := \frac{1}{C}\Delta_{i}(\frac{\varepsilon}{C}, p_0, t),
	\end{equation}
	then 
	\begin{equation}
	d(p_0, q) \leq \hat{\delta_i}(\varepsilon, p_0, t) \implies d(H_{i+1}(p_0, t),H_{i+1}(q,t))\leq \varepsilon.
	\end{equation}
	So given a point $p_0$ such that $B_{\delta_2}(p_0)$ intersects $\bigsqcup\limits_{\gamma \in \Gamma} B_{\delta_2}(\gamma \tilde{x}_i)$, then 
	\begin{equation}\label{finesa}
	  d((p_0,t), (q,s)) \leq \delta_{i+1}(\varepsilon, p_0, t) \implies d(H_{i+1}(p_0, t),H_{i+1}(q,s))\leq \varepsilon.
	\end{equation}
	where 
	\begin{equation}
	    \begin{split}
	    \delta_{i+1}(\varepsilon, p_0, t) :&= \min \{\hat{\delta_i}(\frac{\varepsilon}{2}, p_0, t), \tilde{\delta_i}(\frac{\varepsilon}{2}, p_0, t) \} \\
	    &= \min \{\frac{1}{C^2}\delta_i(\frac{\varepsilon}{4C^2}, p_0, t), \frac{1}{C^2}\delta_i(\frac{\varepsilon}{8C^2}, p_0, t)\} \\
	    &=\frac{1}{C^2}\delta_i(\frac{\varepsilon}{8C^2}, p_0, t),
	    \end{split}
	\end{equation}
	indeed
	\begin{equation}
	\begin{split}
	    d(H_{i+1}(p_0, t),H_{i+1}(q,s)) &\leq d(H_{i+1}(p_0, t),H_{i+1}(p_0,s)) + d(H_{i+1}(p_0, s),H_{i+1}(q,s))\\ &\leq \frac{\varepsilon}{2} + \frac{\varepsilon}{2} = \varepsilon.
	\end{split}
	\end{equation}
	Then, since a ball $B_{\delta_2}(p_0)$ intersects $\bigsqcup\limits_{\gamma \in \Gamma} B_{\delta_2}(\gamma \tilde{x}_i)$ at most $K$ times, we obtain that there is a $\delta(\varepsilon)$ which is given by 
	\begin{equation}
	 \delta(\varepsilon) := \frac{1}{C^{2K}}\delta_0(\frac{\varepsilon}{8^K C^{2K}})
	\end{equation}
	 such that for each $(p,t)$ and $(q,s)$ in $M \times [0,1]$, we have
	\begin{equation}
	d((p,t),(q,s)) \leq \delta(\varepsilon) \implies d(H(p,t), H(q,s)) \leq \varepsilon.
	\end{equation}
\end{proof}
\begin{cor}\label{corappr}
	Let $f: (M,g) \longrightarrow (N,h)$ be a uniformly continuous map between manifolds of bounded geometry. Assume the existence of a closed set $C$ such that $f_{\vert_{M \setminus C}}$ is a $C^k_{b}$-map. Then for all $\epsilon > 0$ there is a map $F:(M,g) \longrightarrow (N,h)$ such that $F$ is a $C^k_{b}$-map. Moreover if $C_\epsilon$ is the $\epsilon$-neighborhood of $C$, then 
	\begin{equation}
		F_{\vert_{M \setminus C_\epsilon}} = f_{\vert_{M \setminus C_\epsilon}}.
	\end{equation}
	Finally if $C$ is $\Gamma$-invariant and $f$ is $\Gamma$-equivariant then also $F$ is $\Gamma$-equivariant and they are $\Gamma$-uniformly homotopic.
\end{cor}
\begin{cor}\label{C^k_bhomo}
	Let $f, f': (M,g) \longrightarrow (N,h)$ be two smooth $C^k_b$-maps such that $f \sim_\Gamma f'$, where $\Gamma$ acts FUPD on $M$ and $N$ by isometries. Let us fix some normal coordinates $\{x^i, s\}$ on $M \times [0,1]$ and $\{y^j\}$ on $N$. Then there is a uniformly continuous homotopy $H: (M \times [0,1], g + dt^2) \longrightarrow (N,h)$ between $f$ and $f'$ such that all its derivatives in normal coordinates of order minor or equal of $k$ are uniformly bounded.
\end{cor}
\begin{cor}\label{puppa}
		Let $f:(M,g) \longrightarrow (N,h)$ be a uniform homotopy equivalence between manifolds of bounded geometry. Then $f$ is uniformly proper and, in particular, $f$ is a uniform map.
	\end{cor}
	\begin{proof}
		Let us denote by $g$ a uniform homotopy inverse of $f$ and by $H$ a uniformly continuous homotopy between $g \circ f$ and $id_X$. Because of Proposition \ref{appr} and Remark \ref{uniformlyp}, we can assume that $f$, $g$ and $H$ are Lipschitz maps. Then, if $x$ is a point in $M$,
		\begin{equation}
		d_X(x, g\circ f(x)) = d_X(H(x,0),H(x,1)) \leq C_H d_{X\times[0,1]}((x,0), (x,1)) = C_H
		\end{equation}
		where $C_H$ is the Lipschitz constant of $H$. Fix a subset $A$ of $N$. Let $x_1$ and $x_2$ be two points in $f^{-1}(A)$. Then \begin{equation}
		\begin{split}
		d_X(x_1,x_2) &\leq d_X(x_1, g\circ f (x_1)) + d_X( g\circ f (x_1),  g\circ f (x_2)) + d_X(x_1, g\circ f (x_1))\\
		&\leq 2C_H + C_gd_Y(f(x_1), f(x_2))\\
		&\leq 2C_H + C_gdiam(A)
		\end{split}
		\end{equation}
		where $C_g$ is the Lipschitz constant of $g$. And so we obtain
		\begin{equation}
		    diam(f^{-1}(A)) \leq 2C_H + C_gdiam(A).
		\end{equation}
	\end{proof}
\section{$L^{q,p}$-cohomology and pull-back}
\subsection{$L^p$-forms}
Let us consider a Riemannian manifold $(M,g)$ and let us denote by $\Omega^k_{c}(M)$ the space of complex differential forms with compact support. Fix a $p \in [1, + \infty)$. We can define a norm on $\Omega^*_c(M)$ as follows:
\begin{equation}
\vert\vert \alpha \vert\vert_{p} := [\int_M \vert\alpha\vert^p(q) d\mu_M(q) ]^\frac{1}{p},
\end{equation}
where $\mu_M$ is the measure on $M$ induced by $g$.
\begin{defn}
	Fix $ p \in [1, +\infty)$. We denote by $L^p\Omega^k(M)$ the Banach space given by the closure of $\Omega^*_{c}(M)$ with respect to the norm $\vert\vert \cdot \vert\vert_{p}$. Moreover we can also define the Banach space $\mathcal{L}^p(M):= \bigoplus \limits_{k \in \mathbb{N}} L^p\Omega^k(M)$.
\end{defn}
\begin{defn}
Consider two Riemannian manifolds $(M,g)$ and $(N,h)$. Let us consider a linear operator $A: \Omega^*(N) \longrightarrow \Omega^*(N)$. The operator $A$ is $\mathcal{L}^\star$-\textbf{bounded} if for each $p$ in $[1, + \infty)$ the operator $A_{\vert_{\Omega^*_c(M)}}$ is a bounded operator with respect to the $\mathcal{L}^p$-norm.
\end{defn}
\subsection{$L^{q,p}$-cohomology and reduced $L^{q,p}$-cohomology}
In the next section we use definitions and results from the work of  Gol'dshtein and Troyanov \cite{Gol} and from the work of Bei \cite{Bei2}. Consider a complete Riemannian manifold $(M,g)$. Then, for each choice of $q,p \in [1, +\infty)$ and $k$, we need a closed extension of the exterior derivative operator 
\begin{equation}
d_{q,p}^{k-1}: \Omega^{k-1}_c(M) \subset \mathcal{L}^q\Omega^{k-1}(M,g) \longrightarrow \mathcal{L}^p\Omega^{k}(M,g).
\end{equation}
\begin{defn}
The \textbf{minimal extension} $d_{min,q,p}^{k-1}$ of the exterior derivative $d_{q,p}^{k-1}$ has domain given by the $\mathcal{L}^q$-forms $\alpha$ such that there is a sequence of compactly supported differential forms $\{\alpha_k\}$ and an $\mathcal{L}^p$-form $\beta$ such that\footnote{For each $p$ in $[1, +\infty)$ we consider on $L^p\Omega^k(M)$ the topology induced by $\vert\vert \cdot \vert\vert_{p}$} $\alpha = \lim\limits_{k \rightarrow + \infty} \alpha_k$ and $\beta = \lim\limits_{k \rightarrow + \infty} d\alpha_k$.
\\Then we obtain
\begin{equation}
d_{min,q,p}^{k-1} \alpha := \lim\limits_{k \rightarrow + \infty} d\alpha_k.
\end{equation}
If there is no need to specify $q, p$ and $k$ we will just denote it by $d$.
\end{defn}
\begin{rem}
Observe that $d_{min,p,s}^{k-1}(dom(d_{min,p,s}^{k-1})) \subseteq dom(d_{min,s,q}^{k})$ and $d^2 := d_{min,s,q}^{k} \circ d_{min,p, s}^{k-1} = 0$ for each choice of $p,s,q$ and $k$.
\end{rem}
\begin{rem}
Because of the lower bound on the injectivity radius, every manifold of bounded geometry is a complete Riemannian manifold. As a consequence of this fact, $d_{min,p,q}$ is the unique closed extension of the exterior derivative for each choice of $p,q \in [1, + \infty)$. The proof is exactly the same of Proposition 3.1 of \cite{Bei2}, indeed the arguments used by the author and Theorem 12.5 of \cite{Gol} also hold if $q \neq p$.
\end{rem}
Consider $p,q$ in $[1, +\infty)$.
\begin{defn}
	The \textbf{$i$-th group of $L^{q,p}$-cohomology } is the group
	\begin{equation}
	H^i_{q,p}(M) := \frac{ker(d^i_{p,p})}{im(d^{i-1}_{q,p})}.
	\end{equation}
	The \textbf{$i$-th group of reduced $L^{q,p}$-cohomology } is the group
	\begin{equation}
	\overline{H}^i_{q,p}(M) := \frac{ker(d^i_{p,p})}{\overline{im(d^{i-1}_{q,p})}}.
	\end{equation}
	Finally, the \textbf{$i$-th group of $L^{q,p}$-quotient cohomology} is the group
	\begin{equation}
	T^i_{q,p}(M) := \frac{\overline{im(d^{i-1}_{q,p})}}{im(d^{i-1}_{q,p})}.
	\end{equation}
\end{defn}
\begin{rem}
The $i$-th group of $L^{q,p}$-quotient group is also called $i$-th group of $L^{q,p}$-torsion by Gol’dshtein and Kopylov in page 4 of \cite{Gol}.
\end{rem}
\begin{prop} \label{boundedness}
		Let $(N,h)$ and $(M,g)$ be two oriented Riemannian manifolds and let $p \in [1, +\infty)$. Fix $b \in \mathbb{Z}$ and consider $B: \Omega^*(N) \longrightarrow \Omega^{* + b}(M)$ and $K: \Omega^*(N) \longrightarrow \Omega^{* + b -1}(M)$ two $\mathcal{L}^*$-bounded operators. Let us suppose that $B(\Omega^*_c(N)) \subseteq \Omega^{* + b}_c(M)$ and
		\begin{equation}
		d B = \pm Bd + K
		\end{equation}
		over $\Omega^*_c(N)$. Then $B(dom(d_{min,q,p})) \subseteq dom(d_{min,q,p})$ and $d_{min,q,p}B = \pm Bd_{min} + K$ on the minimal domain of $d$.
	\end{prop}
	\begin{proof}
		Let $\alpha$ be an element in $dom(d_{min, q,p})$. This means that there is a sequence $\{\alpha_n \}$ in $\Omega^*_c(N)$ such that $\alpha = \lim\limits_{n \rightarrow +\infty} \alpha_n$ and $d\alpha = \lim\limits_{n \rightarrow +\infty} d\alpha_n$. Since $B$ is continuous we obtain
		\begin{equation}
		B\alpha = \lim\limits_{n \rightarrow +\infty} B\alpha_n
		\end{equation}
		where $\{B\alpha_n\}$ is a sequence in $\Omega^*_c(M)$. Moreover the limit of $dB\alpha_n$ exists, indeed
		\begin{equation}
			\begin{split}
		\lim\limits_{n\rightarrow + \infty} dB\alpha_n &= \lim\limits_{n\rightarrow + \infty} \pm Bd\alpha_n + K\alpha_n\\
		&= \pm B\lim\limits_{n\rightarrow + \infty} d\alpha_n + K\lim\limits_{n\rightarrow + \infty}\alpha_n = \pm Bd\alpha + K\alpha.
			\end{split}
		\end{equation}
		So $d_{min,q,p}$ is well defined in $B\alpha$ and
		\begin{equation}
		d_{min,q,p}B \alpha = \pm Bd_{min,q,p}\alpha + K\alpha.
		\end{equation}
	\end{proof}
	\begin{rem}\label{altra}
		Consider an $\mathcal{L}^*$-bounded operator $A$ such that $A(dom(d_{min,N})) \subseteq dom(d_{min,M})$. Suppose that $Ad = dA$. Then $A$ induces a map in reduced $L^{p,q}$-cohomology and in $L^{q,p}$-quotient cohomology. Indeed, given $\alpha$ in $dom(d_N)$, we have 
		\begin{equation}
		\begin{split}
		&A(\alpha + \lim\limits_{k \rightarrow + \infty}d\beta_k) = A(\alpha) + A(\lim\limits_{k \rightarrow + \infty}d\beta_k) \\
		&=A(\alpha) + \lim\limits_{k \rightarrow + \infty} A(d \beta_k) =A(\alpha) + \lim\limits_{k \rightarrow + \infty} dA(\beta_k).
		\end{split}
		\end{equation}
		So, in reduced $L^{q,p}$-cohomology and in $L^{q,p}$-quotient cohomology, we obtain
		\begin{equation}
		[A(\alpha + \lim\limits_{k \rightarrow + \infty}d\beta_k)] = [A(\alpha)].
		\end{equation}
	\end{rem}
	\begin{cor}\label{boundedness2}
		Consider two operators $B$ and $K$ as in Proposition \ref{boundedness}.
		\\Then  $K$ induces the null operator in (un)-reduced $L^{q,p}$-cohomology and in $L^{q,p}$-quotient cohomology. \\Moreover, if $K = 0$ as operator between the $\mathcal{L}^p$-spaces, then  $dB = Bd$ on $\Omega_c^*(M)$. Then $B$ induces a map in (un)-reduced $L^{q,p}$-cohomology and in $L^{q,p}$-quotient cohomology.\footnote{We are considering $\overline{d} = d_{min}$.}
	\end{cor}
\subsection{Fiber Volume and Radon-Nikodym-Lipschitz maps}
Let $(M, \nu)$ and $(N, \mu)$ be two measured spaces and let $f: (M, \nu) \longrightarrow (N, \mu)$ be a function such that the pushforward measure $f_\star(\nu)$ is absolutely continuous with respect to $\mu$.
\begin{defn}
	Let $(N, \mu)$ be $\sigma$-finite, then the \textbf{Fiber Volume} is the Radon-Nikodym derivative
	\begin{equation}
	Vol_{f, \nu, \mu} := \frac{\partial f_\star \nu}{\partial \mu}.
	\end{equation}
\end{defn}
Consider $(M, d_M, \mu_M)$ and $(N, d_N, \mu_N)$ two measured and metric spaces. 
\begin{defn}
	A map $f: (M, d_M, \mu_M) \longrightarrow (N, d_N, \mu_N)$ is \textbf{Radon-Nikodym-Lipschitz} or \textbf{R.-N.-Lipschitz} if
	\begin{itemize}
		\item $f$ is Lipschitz
		\item $f$ has a well-defined and bounded Fiber Volume.
	\end{itemize}
\end{defn}
\begin{rem}
	Consider $f: (M, d_M, \nu_M) \longrightarrow (N, d_N, \mu_N)$ an R.-N.-Lipschitz map and let $C$ be the supremum of $\frac{\partial f_{\star} \mu_M}{\partial \mu_N}$. Then for all measurable set $A \subseteq N$, 
	\begin{equation}\label{follow}
	\mu_M(f^{-1}(A)) = \int_A \frac{\partial f_{\star} \mu_M}{\partial \mu_N} d\mu_N \leq C \int_A d\mu_N = C \mu_N(A).
	\end{equation}
The vice-versa also holds: if $f$ satisfies (\ref{follow}), then it is a R.N.-Lipschitz map. This implication can be proved using a \textit{reductio ad absurdum} argument.
	\end{rem}
	\begin{rem}\label{compo2}
	Composition of R.-N.-Lipschitz maps is a R.-N.-Lipschitz map.
	\end{rem}
\begin{lem}
	Let $f:(M,g) \longrightarrow (N,h)$ be a Lipschitz map between Riemannian manifolds. Let $x$ be a point on $M$ such that $f$ is differentiable on $x$. Then there is a number $C$, which does not depend on $x$, such that for all $\alpha_{f(x)}$ in $T^*_{f(x)}(N)$ and for all $k \in [1, +\infty)$ 
	\begin{equation}
	\vert f^*\alpha_{f(x)} \vert^k_{x} \leq C^k \cdot \vert\alpha_{f(x)} \vert^k_{f(x)},
	\end{equation}
	where $\vert \cdot \vert_{x}$ and $\vert \cdot \vert_{f(x)}$ are the norms induced by the metrics $g$ and $h$ on $\Lambda^*_xM$ and $\Lambda^*_{f(x)}N$.
\end{lem}
\begin{proof}
It follows since the norms on $\Lambda^*_pN$ and $\Lambda^*_qM$ can be seen as operatorial norms and since the $df_q$ has bounded norm.  
\end{proof}
\begin{rem}
By the Rademacher Theorem, we know that a Lipschitz map is differentiable almost everywhere and so the previous lemma holds for almost all $x$ in $M$.
\end{rem}
\begin{prop}
	Let $(M,g)$ and $(N,h)$ be Riemannian manifolds. Let $f:(M,g) \longrightarrow (N,h)$ be a R.-N.-Lipschitz map. Then $f$ induces an $\mathcal{L}^*$-bounded pullback. This means that if $f$ is a R.N.-Lipschitz map, then $f^*$ is an $\mathcal{L}^*$-bounded operator.
\end{prop}
\begin{proof}
	Let $\omega$ be a smooth form with compact support in $\mathcal{L}^p(N)$ and let $K_f := \max\{1, C_f^n\}$, where $n = dim(N)$. Then
	\begin{equation}
	\begin{split}
	\vert\vert f^*\omega\vert\vert_{p}^p &= \int_M \vert f^*\omega \vert^p d\mu_M 
	\leq \int_M K_f^p f^*(\vert \omega \vert^p) d\mu_M \\
	&= K_f^p \int_N \vert \omega \vert^p d(f_\star \mu_M)
	= K_f^p \int_N \vert \omega \vert^p Vol_{f,\mu_M,\mu_N} d\mu_N\\
	&\leq K_f^p C_{Vol}  \int_N \vert \omega \vert^p d\mu_N
	=  K_f^p C_{Vol} \vert\vert\omega\vert\vert_{p}^p.
	\end{split}
	\end{equation}
\end{proof}
In the next sections we will focus on submersions. In particular, we will study their Fiber Volumes. In order to do this we need the notion of quotient of differential forms.
	\begin{defn}
		Let us consider a differentiable manifold $M$. Given two differential forms $\alpha \in \Omega^k(M)$, $\beta \in \Omega^n(M)$ we define \textbf{a quotient between $\alpha$ and $\beta$}, denoted by $\frac{\alpha}{\beta}$, as a (possibly not continuous) section of $\Lambda^{k-n}(M)$ such that for all $p$ in $M$
		\begin{equation}
		\alpha(p) =  \beta(p) \wedge \frac{\alpha}{\beta}(p).\label{quotient}.
		\end{equation}
	\end{defn}
Consider two oriented differentiable manifolds $X$ and $Y$. Given a submersion $f:X \longrightarrow Y$, let us denote by $F_q$ the fiber of $f$ in $q$ and consider $i_q: F_q \longrightarrow X$ the immersion of the fiber in $X$. Then $i_q^*(\frac{\beta}{f^*\alpha})$ does not depend on the choice of the quotient. This is proved in Proposition 16.21.7 of the book of Dieudonn\'e \cite{Dieu}. Moreover in \cite{Dieu} it is also proved that if $\beta$ is a smooth form in $X$, then $i_q^*(\frac{\beta}{f^*\alpha})$ is a smooth form on the fiber $F_q$.
	This means that for all $p$ in $Y$ we obtain an orientation of $F_p$ defined setting
	\begin{equation}
	\int_{F_p} i_p^*\frac{Vol_X}{f^*Vol_Y} > 0.
	\end{equation}
Observe that if $f:(X,g) \longrightarrow (Y,h)$ is a submersion between Riemannian manifolds, then it is possible to define a \textit{locally smooth} quotient between $Vol_X$ and $f^*Vol_Y$. Indeed it is sufficient to consider local fibered coordinates $\{x^i, y^j\}$ on $X$ and $\{y^j\}$ in $Y$ and we obtain that locally
\begin{equation}
	\frac{Vol_X}{f^*Vol_Y} = \frac{det(g_{ij}(y,x))}{det(h_{rs})(y)}dx^1 \wedge ... \wedge dx^n,
\end{equation}
where $g_{ij}$ and $h_{rs}$ are the matrix related to the metrics $g$ and $h$. As a consequence of this fact the Projection Formula holds also for quotients of volume forms, i.e. given a differential form $\alpha$ in $Y$
\begin{equation}
\int_F f^*\alpha \wedge \frac{Vol_X}{f^*Vol_Y} = \alpha \wedge \int_F \frac{Vol_X}{f^*Vol_Y}.
\end{equation}
In order to prove this we have to consider a cover of coordinate charts and to decompose the integral using a partition of unity. Remember that, as a consequence of Proposition 16.21.7 of \cite{Dieu}, the integral over the fiber in $p$ of a quotient between a form $\alpha$ and a form $f^*\beta$ does not depend on the choice of the quotient. This means that in each chart we can choose a smooth quotient and apply the usual Projection Formula.
\subsection{Fiber Volume of a submersion}
In this section we will study the Fiber Volumes of Lipschitz submersions between orientable manifolds. 
\begin{prop}\label{cosa}
	Let $(M,g)$ and $(N,h)$ two oriented, Riemannian manifolds possibly with boundary. Let $\pi: (M,g) \longrightarrow (N,h)$ be a submersion. Then
	\begin{equation}
	Vol_{\pi,\mu_M,\mu_N}(q) = \int_F \frac{Vol_M}{\pi^*Vol_N}(q).
	\end{equation}
\end{prop}
\begin{proof}
	Let $A$ a measurable set of $N$. Then
	\begin{equation}
	\begin{split}
	f_\star \mu_M (A) &= \int_{\pi^{-1}(A)} 1 d\mu_{M} =\int_{\pi^{-1}(A)} Vol_{M} \\
	&= \int_{\pi^{-1}(A)} \pi^*(Vol_N)\wedge \frac{Vol_M}{\pi^*(Vol_N)} \\
	&= \int_{\pi^{-1}(A \cap \pi(M))} \pi^*(Vol_N)\wedge \frac{Vol_M}{\pi^*(Vol_N)} \\
	&= \int_{A \cap \pi(M)} Vol_N (\int_F \frac{Vol_M}{\pi^*(Vol_N)})  \\
	&= \int_{A \cap \pi(M)} (\int_F \frac{Vol_M}{\pi^*(Vol_N)}) Vol_N \\
	&= \int_{A \cap \pi(M)} (\int_F \frac{Vol_M}{\pi^*(Vol_N)}) d\mu_N.
	\end{split}
	\end{equation}
\end{proof}
\begin{rem}\label{oss}
	If the submersion $f: X \rightarrow Y$ is a diffeomorphism between oriented manifold which preserves the orientations, then the integration along the fibers of $f$ is the pullback $(f^{-1})^*$. This means that the Fiber Volume of $f$ is given by $\vert(f^{-1})^*\frac{Vol_X}{f^*(Vol_Y)}\vert$.
\end{rem}
We conclude this section by giving a formula which allow us to compute the Fiber Volume of the composition of two submersions.
\begin{prop}\label{compo}
	Let $f:(M,g) \longrightarrow (N,h)$ and $g:(N,h) \longrightarrow (W, l)$ be two submersions between oriented Riemannian manifolds. Then
	\begin{equation}
	Vol_{g \circ f, \mu_M, \mu_W}(q) = \int_{g^{-1}(q)}(\int_{f^{-1}g^{-1}(q)} \frac{Vol_M}{f^*(Vol_N)})\frac{Vol_N}{g^*Vol_N}
	\end{equation}
\end{prop}
\begin{proof}
	Observe that, as quotients,
	\begin{equation}
	\frac{Vol_M}{(g \circ f)^*Vol_W} = \frac{Vol_M}{f^*Vol_N} \wedge \frac{f^*Vol_N}{(g \circ f)^*Vol_W}
	\end{equation}
	and, in particular, we can choose as quotient
	\begin{equation}
	\frac{f^*Vol_N}{(g \circ f)^*Vol_W} = f^*(\frac{Vol_N}{g^*Vol_W}).
	\end{equation}
	Then we conclude by applying the Projection Formula:
	\begin{equation}
	\begin{split}
	&Vol_{g \circ f, \mu_M, \mu_W}(q) = \int_{(g \circ f)^{-1}(q)} \frac{Vol_M}{(g \circ f)^*Vol_W}\\
	&= \int_{(g \circ f)^{-1}(q)} \frac{Vol_M}{f^*Vol_N} \wedge f^*(\frac{Vol_N}{g^*Vol_W}) = \int_{g^{-1}(q)}(\int_{f^{-1}g^{-1}(q)} \frac{Vol_M}{f^*(Vol_N)})\frac{Vol_N}{g^*Vol_N}.
	\end{split}
	\end{equation}
\end{proof}
	\section{Sasaki metric and submersions}
	\subsection{The Sasaki metric}\label{sasaki}
	The definition of \textit{Sasaki metric} on a vector bundle given in this paper is a generalization of the metric defined by Sasaki in \cite{SSK} for the tangent bundle of a Riemannian manifolds. The definition of this generalized Sasaki metric can be found in page 2 of the paper of Boucetta and Essoufi \cite{Boucetta}.
	\\Let us consider a Riemannian manifold $(N,h)$ of dimension $n$, $\pi_E : E \longrightarrow N$ a vector bundle of rank $m$ endowed with a bundle metric $H_E \in \Gamma(E^*\otimes E^*)$ and a linear connection $\nabla_E$ which preserves $H_E$. Fix $\{s_\alpha \}$ a local frame of $E$: if $\{x^i\}$ is a system of local coordinates over $U \subseteq N$, then we can define the system of coordinates $\{x^i, \mu^\alpha \}$ on $\pi_E^{-1}(U)$, where the $\mu^\alpha $ are the components with respect to $\{s_\alpha \}$.
	\\Let us denote by $K$ the map $K: TE \longrightarrow E$ defined as
	\begin{equation}
	K(b^i\frac{\partial}{\partial x^i }\vert_{(x_0, \mu_0)} + z^\alpha \frac{\partial}{\partial \mu^\alpha }\vert_{(x_0, \mu_0)}) := (z^\alpha + b^i \mu^j \Gamma_{ij}^\alpha (x_0))s_\alpha (x_0),
	\end{equation}
	where the $\Gamma_{ij}^l$ are the Christoffel symbols of $\nabla_E$. The Christoffel symbols $\Gamma^\gamma_{\eta j}(x)$ are defined by the formula
		\begin{equation}
		 \nabla^E_{\frac{\partial}{\partial x^j }} s_{\eta}(x) := \Gamma^\gamma_{\eta j}(x) s_\gamma(x).
		\end{equation}
	\begin{defn}
		The \textbf{Sasaki metric} on $E$ is the Riemannian metric $h^E$ defined for all $A,B$ in $T_{(p, v_p)}E$ as
		\begin{equation}
		h^E(A,B) := h(d\pi_{E,v_p}(A), d\pi_{E,v_p}(B)) + H_E(K(A), K(B)).
		\end{equation}
	\end{defn}
	\begin{rem}\label{riemsub}
		Let us consider the system of coordinates $\{x^i\}$ on $N$ and $\{x^i, \mu^j\}$ on $E$. The components of $h^E$ are given by
		\begin{equation}\label{metri}
		\begin{cases}
		h^E_{ij}(x,\mu) = h_{ij}(x) + H_{\alpha\gamma}(x)\Gamma^\alpha_{\beta i}(x)\Gamma^\gamma_{\eta j}(x)\mu^\beta\mu^\eta \\
		h^E_{i\sigma}(x, \mu) = H_{\sigma \alpha}(x)\Gamma^\alpha_{\beta i}(x)\mu^\beta\\
		h^E_{\sigma\tau}(x,\mu) = H_{\sigma, \tau}(x), 
		\end{cases}
		\end{equation}
		where $i,j = 1,..., n$ and $\sigma, \tau = n+1,...,n+m$.
		Consider a point $x_0 = (x^1_0, ..., x^n_0)$ in $N$. If all the Christoffel symbols of $\nabla_E$ in $x_0$ are zero, then, in local coordinates, the matrix of $h^E$ in a point $(x_0, \mu)$ is
		\begin{equation}\label{metrinulla}
		\begin{bmatrix}
		h_{i,j}(x_0) && 0 \\
		0 && H_{\sigma, \tau}(x_0)
		\end{bmatrix}.
		\end{equation}
Moreover, with respect to the coordinates $(x^i, \mu^\sigma)$, the matrix $h_E := (h^E)^{-1}$ is given by
\begin{equation}\label{invmetri}
    \begin{cases}
        h^{ij}_E(x,\mu) =  h^{ij}(x) \\
        h^{i\sigma}_E = -\Gamma^\sigma_{\beta j}(x) h^{ij}(x)\mu^\beta \\
        h_E^{\sigma \tau} = H^{\sigma \tau}(x) +  h^{ij}(x) \Gamma^\sigma_{\beta i}(x)\Gamma^\tau_{\eta j}(x)\mu^\beta\mu^\eta
    \end{cases}
\end{equation}
where $H^{\sigma \tau}$ and $ h^{ij}(x)$ are the components of the inverse matrices of $h_{ij}$ and $H_{\sigma, \tau}$.
	\end{rem}
	\begin{exem}\label{twos}
		Let $(M,g)$ a Riemannian manifold. Consider as $E$ the tangent bundle $TM$ and as $h_E$ the metric $g$ itself. Choose the connection $\nabla_E$ as the Levi-Civita connection $\nabla_g^{LC}$. We denote by $g_S$ the Sasaki metric induced by $g$ and $\nabla^{LC}_g$.
	\end{exem}
	\begin{exem}\label{threes}
	Consider a smooth map $f:(M,g) \longrightarrow (N,h)$. Let $\pi: f^*TN \longrightarrow M$ be the pullback bundle. Then the Riemannian metric $h$ can be seen as a bundle metric on $TN$ and so we obtain a bundle metric $f^*h$ on $f^*TN$. Fix the connection $f^*\nabla^{LC}_h$ on $f^*TN$ which is the pullback of the Levi-Civita connection on $(N,g)$. Let us denote by $g_{S,f}$ the Sasaki metric induced by $f^*\nabla^{LC}_h$, $f^*h$ and $g$.
	\end{exem}
	\begin{rem}
	 Then the Christoffel symbols of the pullback connection $f^*\nabla^E$ with respect to the pullback frame $\{f^* e_i\}$ and to the coordinates $\{x^i\}$ are given by 
	\begin{equation}
	 \tilde{\Gamma}^\alpha_{\beta, i}:= \frac{\partial f^l}{\partial x^i}f^*(\Gamma^\alpha_{\beta, l}).
	\end{equation}
	A first consequence of this fact is that the map $ID: (id^*TN, g_{S, id}) \longrightarrow (TN, g_S)$ which sends $(p, w_{p})$ to $w_p$ is an isometry. So, from now on, we will identify $(id^*TN, g_{S, id})$ and $(TN, g_S)$.
	\end{rem}
	\begin{rem}\label{respect}
	Let $f: (M,g) \longrightarrow (N,h)$ be a smooth map. Let us fix a chart $\{U, x^i\}$ on $M$ and a chart $\{V, y^j\}$ on $N$ such that $f(U) \subseteq V$. Fix a bundle $E$ on $N$ and let $\nabla^E$ be a connection. Let us denote by $\Gamma^\alpha_{\beta l}$ the Christoffel symbols of $\nabla^E$ with respect to a frame $\{e_i\}$ and the coordinates $\{y^j\}$. If $f: (M,g) \longrightarrow (N,h)$ is a smooth Lipschitz map, then also the induced bundle map 
	\begin{equation}
	    \begin{split}
	        F:(f^*T^\delta N, g_{S,f}) &\longrightarrow (T^\delta N, g_S) \\
	        (p, w_{f(p)}) &\longrightarrow w_{f(p)}
	    \end{split}
	\end{equation}
	is a smooth Lipschitz map.
	\end{rem}
	\begin{prop}\label{geodesis}
	 Consider a vector bundle $(E, \pi, M)$ over a Riemannian manifold $(M,g)$. Fix on $E$ a bundle metric $h$, a connection $\nabla^E$ and let us denote by $h_E$ the Sasaki metric induced by $g$, $h$ and $\nabla^E$. Let us suppose that for each point $p$ on $M$ there is a system of normal coordinates $\{x^i\}$ around $p$ and a local frame $\{s_\sigma\}$ such that the Christoffel symbols of $\nabla^E$ vanishes at $x= 0$. Moreover let us suppose that $\frac{\partial h_{\sigma \tau}}{\partial x^k}(0) = 0$, where $h_{\sigma \tau}$ are the components of the Gram matrix of $h$ with respect to the coordinates $\{x^i\}$. Then the fibers of $\pi$ are totally geodesic submanifolds (in our case this means that the straight lines on the fibers are geodesics) and $\pi$ is a Riemannian submersion.
	\end{prop}
		\begin{rem}
	The vector bundles of the Examples \ref{twos} and \ref{threes} satisfy the assumptions of Proposition \ref{geodesis}.
	\end{rem}
	\begin{proof}
	Fix the coordinates $\{x^i, \mu^\sigma\}$ where the coordinates $\{x^i\}$ are normal and centered on a point $p$ and $\mu^\sigma$ refer to the local frame $\{s_\sigma\}$. It follows from (\ref{metrinulla}) that $\pi$ is a Riemannian submersion. Let $v_p$ and $w_p$ be two vectors in $E_p$. Let us suppose that $v_p = (0, v^\sigma)$ and $w_p = (0, w^\sigma)$ with respect to the coordinates $\{x^i, \mu^\sigma\}$. Consider the straight line connecting $v_p$ and $v_p$ defined as $\gamma(t) = t\cdot (v_p - w_p) + w_p$. Consider the Christoffel symbols of the Levi-Civita connection induced by $h^E$ with respect to the coordinates $\{x^i, \mu^\sigma\}$ and the local frame $\{ \frac{\partial}{\partial x^i},\frac{\partial}{\partial \mu^\sigma} \}$. Let us denote by $\tilde{\Gamma}^{i}_{\sigma \tau}$ and by $\tilde{\Gamma}^{\eta}_{\sigma \tau}$ the Christoffel symbols defined by
	\begin{equation}
	 \nabla^{h^E}_{\frac{\partial}{\partial \mu^\sigma}} \frac{\partial}{\partial \mu^\tau}(x,\mu) = \tilde{\Gamma}^{i}_{\sigma \tau}(x, \mu) \frac{\partial}{\partial x^i} + \tilde{\Gamma}^{\eta}_{\sigma \tau}(x, \mu)\frac{\partial}{\partial \mu^\eta}.
	\end{equation}
	Because the partial derivatives of $g_{ij}$ and of $h_{\sigma, \tau}$ in $x= 0$ are zero, we obtain that $\tilde{\Gamma}^{j}_{\sigma \tau}(0, \mu) = \tilde{\Gamma}^{\eta}_{\sigma \tau}(0, \mu) = 0$. Then, if we parametrize the curve $\gamma(t) := (0, t(v^\sigma - w^\sigma) + w^\sigma)$, then $\gamma$ satisfies the system of geodesics equations
	\begin{equation}
	\begin{cases}
	\frac{d^2\gamma^{i}}{dt^2}  + \tilde{\Gamma}^{i}_{\sigma \tau}(x, \mu)\frac{d\gamma^{\sigma}}{dt} \frac{d\gamma^{\tau}}{dt} + \tilde{\Gamma}^{i}_{j \tau}(x, \mu)\frac{d\gamma^{j}}{dt} \frac{d\gamma^{\tau}}{dt} + \tilde{\Gamma}^{i}_{jz}(x, \mu)\frac{d\gamma^{j}}{dt}  \frac{d\gamma^{z}}{dt}  = 0 \\
	\frac{d^2\gamma^{\alpha}}{dt^2}  + \tilde{\Gamma}^{\alpha}_{\sigma \tau}(x, \mu)\frac{d\gamma^{\sigma}}{dt} \frac{d\gamma^{\tau}}{dt} + \tilde{\Gamma}^{\alpha}_{j \tau}(x, \mu)\frac{d\gamma^{j}}{dt} \frac{d\gamma^{\tau}}{dt} + \tilde{\Gamma}^{\alpha}_{jz}(x, \mu)\frac{d\gamma^{j}}{dt}  \frac{d\gamma^{z}}{dt}  = 0. 
	\end{cases}
	\end{equation}
	\end{proof}
	\begin{cor}
	Consider $0_E$ the null section of a vector bundle $\pi: E \longrightarrow M$. Under the same assumptions of the Proposition \ref{geodesis}, the disk bundle $E^\delta := \{v_p \in E \vert \sqrt{h(v_p,v_p)}  \leq \delta \}$ coincides with $B_{\delta}(0_E) := \{v_p \in E \vert d_{h^E}(v_p, 0_E) \leq \delta \}$ for each $\delta > 0$.
	\end{cor}
	\begin{proof}
Fix $v_p$ a point on $E$. Because of Proposition \ref{geodesis} we already know that the straight line $\gamma$ which connects $0_p$ and $v_p$ is a geodesic and its length is equal to $\sqrt{h(v_p,v_p)}$. Observe that if there is another point $0_q$ in $0_E$ such that $d(0_q, v_p) < \sqrt{h(v_p,v_p)}$, then there is a geodesics $\sigma$ connecting $v_p$ to $0_q$ which is shorter than $\gamma$.
\\Let us suppose that $\dot{\sigma}(0_q)$ is orthogonal to $T_{0_q}0_E$. Because of Proposition \ref{geodesis}, we know that $q = p$ and $\gamma = \sigma$. This is a contradiction.
\\Let us prove that $\dot{\sigma}(0_q)$ is orthogonal to $T_{0_q}0_E$. We can assume, without loss of generality, that $d(v_p, 0_E)$ is smaller than the radius of injectivity of $E$ in $v_p$.  
Observe that the map $d(v_p, \cdot): 0_E \longrightarrow \mathbb{R}$ is continuous and admits a minimum. \\Consider the ball $B_{d(v_p, 0_E)}(v_p)$. Then
\begin{equation}
B_{d(v_p, 0_E)}(v_p) \cap 0_E \subseteq \partial B_{d(v_p, 0_E)}(v_p) := S_{d(v_p, 0_E)}(v_p)
\end{equation}
and the intersection is the set of the minima of $d(v_p, \cdot)$. In particular, if $q$ is a minimum for $d(v_p, \cdot)$,
\begin{equation}\label{prova}
 T_q0_E \subseteq T_{q}S_{d(v_p, 0_E)}(v_p)
\end{equation}
as subspaces of $T_q E$. In order to prove (\ref{prova}) we use a simple \textit{reductio ad absurdum} argument. Let us suppose there is $v$ in $T_q0_E$ such that $v$ is not in $T_{q}S_{d(v_p, 0_E)}(v_p)$. Without loss of generality we can assume that $v$ is an inward vector for $B_{d(v_p, 0_E)}(v_p)$. On a open neighborhood of $B_{d(v_p, 0_E)}(v_p)$ there are some coordinates $\{x^1, ..., x^{n-1}, y\}$ such that the interior of $B_{d(v_p, 0_E)}(v_p)$ is given by the points $y > 0$. Let $\tau: (-\varepsilon, \varepsilon) \longrightarrow U$ be a curve in $0_E$ such that $\tau(0) = q$ and $\dot{\tau}(0) = v$. Observe that, in coordinates, $v = v^i \frac{\partial}{\partial x^i} + v_0 \frac{\partial}{\partial y}$ where $v_0 \geq 0$. But this means that there is a $t_0$ in $(-\varepsilon, \varepsilon)$ such that $y(\tau(t_0)) > 0$. But $\tau(t_0)$ is a point of $0_E$ and so it cannot be a point on the interior of $B_{d(v_p, 0_E)}(v_p)$. This proves (\ref{prova}).
\\Finally, thanks to the Gauss Lemma, the geodesic $\sigma: [0,1] \longrightarrow E$ connecting $\sigma(0) = v_p$ to a minimum $\sigma(1) = q$ has $\dot{\sigma}(1)$ orthogonal to $T_{q}S_{d(v_p, 0_E)}(v_p) \supseteq T_q0_E$.
\end{proof}
	In order to prove the following Proposition, we need a Lemma which is a classical result of the Measure Theory. A proof of this result can be found in page 29 of the book of Durrett \cite{Durr}.
	\begin{lem}[Multivariate Jensen's Inequality]\label{jensen}
	Let $(X, \Sigma_X, \mu_X)$ be a probability space\footnote{This is a measured space such that $\mu_X(X) =1$} and let $f:X \longrightarrow \mathbb{R}^n$. Consider $\phi: \mathbb{R}^n \longrightarrow \mathbb{R}$ a convex function. Then, for each integrable function $f: X \longrightarrow \mathbb{R}^n$
	\begin{equation}
	 \phi(\int_X f d\mu_X) \leq \int_X \phi(f) d\mu_X.
	\end{equation}
	\end{lem}
	\begin{prop}\label{star}
	Consider a Riemannian manifold $(M,g)$ and let $\pi: E \longrightarrow M$ be a vector bundle. Fix on $E$ a metric bundle $h_E$ and a connection $\nabla_E$. Let $h^E$ be the Sasaki metric on $E$ defined by using $g$, $h_E$ and $\nabla_E$. Fix a $\delta > 0$. Then, under the assumptions of Proposition \ref{geodesis} $\pi: E^\delta \longrightarrow M$ is a R.N.-Lipschitz map. In particular the Fiber Volume on a point $q$ is the Volume of an Euclidean ball of radius $\delta$. Finally the integration along the fibers $\pi_\star: \Omega_{c}^*(E^\delta) \longrightarrow \Omega^*(M)$ is an $\mathcal{L}^*$-bounded operator. 
	\end{prop}
	\begin{proof}
	We know that $\pi$ is a Lipschitz map because it is a Riemannian submersion (Proposition \ref{geodesis}). In order to calculate the Fiber Volume of $\pi$, we choose a point $q$ in $M$ and consider some normal coordinates $\{x^i\}$ around $q$. Fix the frame $\{e_j\}$ of $E$ defined around $q$ such that $\Gamma^i_{jk}(0) = 0$ with respect to $\{e_j\}$ and to $\{x^i\}$. Let  $\{x^i, y^j\}$ be the fibered coordinates on $E^\delta$ where $\{y^j\}$ refer to $e_j$. The matrix related to $h^E$ in a point of the fiber of $q$ is the identity.
	\\This means that $Vol_{E}(x,y) = dx^1 \wedge ... \wedge dx^n \wedge dy^1 \wedge ... \wedge dy^m$ and $\pi^*Vol_M(x) = dx^1 \wedge ... \wedge dx^n$.
	\\We obtain that
	\begin{equation}
	\begin{split}
	Vol_{\pi, \mu_E, \mu_M} &=\int_F \frac{Vol_E}{\pi^*Vol_N} = \int_{B_\delta(0_q)} \frac{dx^1 \wedge ... \wedge dx^n \wedge dy^1 \wedge ... \wedge dy^m}{dx^1 \wedge ... \wedge dx^n}\\
	&= \int_{B_\delta(0_q)} dy^1 \wedge ... \wedge dy^m = Vol(B_\delta(0)).
	\end{split}    
	\end{equation}
	Let us study the $\mathcal{L}^*$-boundedness of $\pi_\star$. First we will show that, given $\alpha$ in $\Omega^*_c(E)$, we have 
	\begin{equation}\label{popo2}
	\vert\pi_\star \alpha\vert_q^p \leq C^{p-1}\int_{E_q} \vert \alpha \vert_y^p d\mu_{E_q}(y),    
	\end{equation}
	where $\mu_{E_q}$ is the measure on the fiber $E_q$ induced by $\frac{Vol_E}{\pi^*Vol_N}$.
	\\Fix the same coordinates $\{x^i, y^j\}$ we used before. We obtain
	\begin{equation}
	 \alpha = \alpha_{I0}(x,y)dx^I \wedge dy^1 \wedge dy^n + \alpha_{IJ}(x,y)dx^I \wedge dy^J,
	\end{equation}
	where $J \neq (1, 2, ..., m)$. Let us denote by $\alpha_0 = \alpha_{I0}(x,y)dx^I \wedge dy^1 \wedge ... \wedge dy^n$. Observe that for each $r$ in the fiber of $q$ we have $\vert\alpha_{0}\vert_r \leq \vert\alpha\vert_r$.
	\\Moreover
	\begin{equation}
	\int_{B_{\delta}(0_q)} \alpha_0 = \int_{B_{\delta}(0_q)} f(y) dy^1 \wedge ... dy^m =\int_{B_{\delta}(0_q)} f(y) d\mu_{E_q}(y)
	\end{equation}
	where $f:B_{\delta}(0_q) \longrightarrow \Lambda^*_q(M) \otimes \mathbb{C}$ and $\vert f(y)\vert = \vert \alpha_0\vert_y$. 
	\\In order to prove (\ref{popo2}), we want to apply the Multivariate Jensen's Inequality (Lemma \ref{jensen}). Let us define for each $q$ in $M$ the measure on the fiber of $q$ defined for each measurable set $A$ of $E_q$ as
	\begin{equation}
	\mu_{E_q,J}(A) := \frac{\mu_{E_q}(A)}{Vol_{\pi, \mu_E, \mu_M}(q)}.
	\end{equation}
	Then $\mu_{E_q,J}$ makes $B_{\delta}(0_q) \subset E_q$ a probability space. Let us define the map $F:B_{\delta}(0_q) \longrightarrow \Lambda^*_qM \otimes \mathbb{C} \cong \mathbb{R}^{2n}$
	\begin{equation}
	F(y) := Vol_{\pi, \mu_E, \mu_M}(q) \cdot f(y),
	\end{equation}
	Then, by the Jensen's Inequality, considering $\phi(x)= \vert x\vert^p$, we obtain that
	\begin{equation}\label{randl}
	\vert \int_{B_{\delta}(0_q)} F(y) d\mu_{E_q,J}\vert^p \leq   \int_{B_{\delta}(0_q)} \vert F(y) \vert^p d\mu_{E_q,J}.
	\end{equation}
	Observe that the left part of (\ref{randl}) is
	\begin{equation}
	    \begin{split}
	        &\vert \int_{B_{\delta}(0_q)} F(y) d\mu_{E_q,J}(y)\vert^p \\
	        &= \vert\int_{B_{\delta}(0_q)}  Vol_{\pi, \mu_E, \mu_M}(q) \cdot f(y) \cdot \frac{1}{Vol_{\pi, \mu_E, \mu_M}(q)} d\mu_{E_q}(y)\vert^p \\
	        &= \vert\int_{B_{\delta}(0_q)}  f(y) d\mu_{E_q}(y)\vert^p = \vert\int_{E_q} \alpha\vert^p = \vert\pi_\star \alpha\vert_q^p.
	    \end{split}
	\end{equation}
	The right part of (\ref{randl}) is
	\begin{equation}
	    \begin{split}
	      \int_{B_{\delta}(0_q)} \vert F(y)\vert^p d\mu_{E_q,J} &=  \int_{B_{\delta}(0_q)}  \vert Vol_{\pi, \mu_E, \mu_M}(q)\vert^p \cdot \vert f(y)\vert^p \frac{1}{Vol_{\pi, \mu_E, \mu_M}(q)} d\mu_{E_q}(y) \\
	      &=\int_{B_{\delta}(0_q)}  \vert Vol_{\pi, \mu_E, \mu_M}(q)\vert^{p-1} \cdot \vert f(y)\vert^p d\mu_{E_q}(y)\\
	      &\leq C^{p-1} \int_{E_q} \vert\alpha\vert_y^p d\mu_{E_q}(y).
	    \end{split}
	\end{equation}
	Then we conclude by showing that
	\begin{equation}
	\begin{split}
	&\vert\vert\pi_\star \alpha\vert\vert_p = \int_M \vert\pi_\star \alpha\vert_q^p d\mu_M(q) \\
	&\leq C^{p-1} \int_M [\int_{E_q} \vert\alpha\vert_y^p d\mu_{E_q}(y)] d\mu_M(q) = C^{p-1} \int_M Vol_M \cdot [\int_{E_q} \vert\alpha\vert_{y_q}^p \frac{Vol_{E}}{\pi^*Vol_M}]\\
	&= C^{p-1} \int_M \int_{E_q} \vert\alpha\vert_y^p \pi^*Vol_M \wedge \frac{Vol_{E}}{\pi^*Vol_M} = C^{p-1} \int_E \vert\alpha\vert_{v}^p Vol_{E}(v) = C^{p-1} \vert\vert\alpha\vert\vert_{p}
	\end{split}
	\end{equation}
	\end{proof}
\begin{rem}\label{Riccib}
Consider $(M,g)$ and $(N,h)$ two manifolds of bounded geometry and let $f:(M,g) \longrightarrow (N,h)$ be a $C^{k}_{b}$-map for each $k$ in $\mathbb{N}$. Fix on $f^*(TN)$ the Sasaki metric $g_{Sf}$ induced by $g$, $f^*h$ and $f^*\nabla^{LC}_h$. If we denote by $\nabla: = \nabla^{LC}_{g_S}$ and by $R$ the Riemann tensor on $f^*(TN)$, we obtain that for each $i$ in $\mathbb{N}$ there is a continuous function $C_i: \mathbb{R} \longrightarrow \mathbb{R}$ such that $\vert\nabla^i R(v_p)\vert \leq C_i(\vert\vert v_p\vert\vert)$. This is a consequence of (\ref{metri}), of (\ref{invmetri}) and of Theorem 2.5 of the paper of Schick \cite{flow}.
\end{rem}
\subsection{A submersion related to a uniform map}
Let $(N,h)$ and $(M,g)$ be two manifolds of bounded geometry and let $\delta \leq inj_N.$ Fix, moreover, a smooth map $f: (M,g) \longrightarrow (N,h)$.
\\In this subsection we define a submersion $p_f: (f^*(T^\delta N), g_{Sf}) \longrightarrow (N,h)$ where 
\begin{equation}
	f^*(T^\delta N) :=\{(p, w_{f(p)}) \in f^*(TN) \mbox{     such that     } \vert w_{f(p)} \vert \leq \delta\}.
\end{equation} 
\begin{lem}\label{tilde}
	Let us consider $f:(M,g) \longrightarrow (N,h)$ a smooth Lipschitz map between two oriented Riemannian manifolds. Suppose that $(N,h)$ is a manifold of bounded geometry.
	\\Let us denote by $F: f^*TN \longrightarrow TN$ the bundle morphism induced by $f$, i.e. $F(p, w_{f(p)}) := w_{f(p)}$. Fix on $f^*TN$ the Sasaki metric $g_{Sf}$ induced by $f^*\nabla^{LC}_h$, $f^*h$ and $g$. Let us denote by $\pi: f^*T^\delta N \longrightarrow M$ be the projection of the bundle. Then there is a map $p_f: (f^*(T^\delta N), g_{Sf}) \longrightarrow (N,h)$ such that:
	\begin{enumerate}
		\item $p_f$ is a submersion,
		\item ${p}_f(x,0) = f(x)$,
		\item ${p}_f$ is $\Gamma$-equivariant\footnote{We are considering on $f^*TN$ the action of $\Gamma$ given by $\gamma \cdot(p, w_{f(p)} ):= (\gamma \cdot p, d\gamma (w_{f(p)}))$},
		\item ${p}_f = {p}_{id_N} \circ F$,
		\item Assume that for each point $p$ in $M$ there are some local coordinates $\{U, x^i\}$ around a point $p$ in $M$ and some normal coordinates $\{V, y^j\}$ around $f(p)$ on $N$ such that $V$ contains a $\delta$-neighborhood of $f(U)$. Assume, moreover, that for each $k$ in $\mathbb{N}$
		\begin{equation}\label{boba}
		\sup\limits_{x \in U }	\sup\limits_{s = 0,...,k} \vert \frac{\partial^s y^j \circ f}{\partial x^i_1 ... \partial x^{i_s}}(x)\vert \leq L_k
		\end{equation}
		for some $L_k$ which does not depend on the choice of $p$. Consider the frame $\{\frac{\partial}{\partial y^j}\}$ around $f(0)$ and define the fibered coordinates $\{x^i, \mu^j\}$ related to $\{\frac{\partial}{\partial y^j}\}$ on $f^*T^\delta N$. Then for each $k$ in $\mathbb{N}$ there is a constant $C_k$ such that 
		\begin{equation}
			\sup\limits_{(x, \mu) \in \pi^{-1}(U) } \sup\limits_{s + t = 0,...,k}\vert \frac{\partial^s y^j \circ p_f}{\partial x^i_1 ... \partial x^{i_s}\partial \mu^{j_1} ... \partial \mu^{j_t}}(x,\mu) \vert \leq C_k
		\end{equation}
		where $C_k$ only depends on the bounds $L_1 ,..., L_k$. In particular if $f$ is a Lipschitz map, then also $p_f$ is a Lipschitz map.
	\end{enumerate}
\end{lem}
\begin{proof}
	Let us define
	\begin{equation}
		\begin{split}
			p_f: (f^*(T^\delta N), g_s) &\longrightarrow (N,h) \\
			(p, w_{f(p)}) &\longrightarrow exp_{f(p)}(w_{f(p)}).
		\end{split}
	\end{equation}
	Then
	\begin{enumerate}
		\item $p_f$ is a submersion. Fix $p$ in $M$. Then ${p}_f(p, \cdot): f^*(T^\delta N)_p = T^\delta_{f(p)} N \longrightarrow N$ is the exponential map in $f(p)$. We know that the exponential map is a local diffeomorphism and so $p_f$ is a submersion,
		\item $p_f(p, 0_{f(p)}) = f(p)$. This follows by the definition of exponential map,
		\item $p_f$ is $\Gamma$-equivariant. Recall that $\Gamma$ acts by isometries. Then
		\begin{equation}
			\begin{split}
				p_f(\gamma p, d\gamma w_{f(p)}) &= exp_{f(\gamma p)}d\gamma w_{f(p)} = exp_{\gamma f(p)}d\gamma w_{f(p)}\\
				&= \gamma exp_{f(p)}w_{f(p)} = \gamma {p}_f(p, w_{f(p)}).
			\end{split}
		\end{equation}
		\item ${p}_f = {p}_{id_N} \circ F$. It is obvious, indeed $F(p, w_{f(p)}) = w_{f(p)}$ and $p_{id}: T^\delta N \longrightarrow N$ is $p_{id}(v_p) := exp_p(v_p)$,
		\item Because of the previous point and because of (\ref{boba}) it is sufficient to prove the assertion only for $p_{id}$. Moreover, in the case of the identity, we can also suppose that $\{x^i\}$ and $\{y^j\}$ are the same normal coordinates.
		\\Consider $p_{id}$ restricted to $\pi^{-1}(U)$. It can be seen as $\pi \circ \phi (x,\mu)$ where $\phi$ is the flow of the system of differential equations given by
		\begin{equation}
			\begin{cases}
				\dot{x}^k = \mu^k\\
				\dot{\mu}^k = - \Gamma^k_{ij}(x) \mu^i\mu^j
			\end{cases}
		\end{equation} 
		Because of Lemma 3.4 of the paper of Schick \cite{flow}, the partial derivatives of $\phi$ are uniformly bounded. Then we conclude that the derivatives of  $p_{id}$ are uniformly bounded. 
	\end{enumerate}
\end{proof}
	\section{The pull-back functor}\label{Pullback}
	\subsection{The Fiber Volume of $p_f$}
	\begin{lem}\label{svolta}
		Consider $f:(M,g) \longrightarrow (N,h)$ a smooth Lipschitz map between Riemannian manifolds of bounded geometry. Let $\delta \leq inj_{N}$. Then the map ${t}_f: (f^*(T^\delta N), g_{Sf}) \longrightarrow (M \times N, g + h)$ defined as
		\begin{equation}
		{t}_f(p,w_{f(p)}) = (p, p_f(p, w_{f(p)}))
		\end{equation}
		is an R.-N.-Lipschitz diffeomorphism with its image.
	\end{lem}
	\begin{proof}
		We start by proving that ${t}_f$ is a diffeomorphism with its image. Observe that 
		\begin{equation}
		dim(f^*(T^\delta N)) = m+ n = dim(M) + dim(N) = dim(M \times N).
		\end{equation}
		 Fix some normal coordinates $\{x^i\}$ around a point $p$ in $M$ and let $\{y^j\}$ be some normal coordinates around $f(p)$ in $N$. Consider the frame $\{\frac{\partial}{\partial y^j}\}$ and define the fibered coordinates $\{x^i, \mu^j\}$ related to $\{\frac{\partial}{\partial y^j}\}$ on $f^*(T^\delta N)$. Consider on $M\times N$ the normal coordinates $\{x^i, y^j\}$. Then the Jacobian of $t_f$ is given by
		\begin{equation}
		Jt_f(x,\mu) = \begin{bmatrix}
		1 && \star \\
		0 && Jexp_{x}(\mu)
		\end{bmatrix}
		\end{equation}
		Then, since the exponential map is a diffeomorphism for each $x_0$,  $Jt_f$ is invertible. Moreover, ${t}_f$ is also injective, indeed if $(p, w_{f(p)})$ and $(q, v_{f(q)})$ have the same image, then $p = q$ and 
		\begin{equation}
		exp_{f(p)} w_{f(p)} = exp_{f(p)}v_p \implies w_p = v_p,
		\end{equation}
		since  their norm is less than $\delta$ and $\delta \leq inj_{N}$. We proved that ${t}_f$ is a diffeomorphism with its image.
		\\ Since $p_f$ is a Lipschitz map, also ${t}_f$ is a Lipschitz map. So, in order to prove that ${t}_f$ is a R.-N.-Lipschitz map, we have to show that it has bounded Fiber Volume. Consider a point $(p,q)$ in $M \times N$. Then its fiber is empty or it is a singleton $\{(p, w_{f(p)})\}$. Because of this and because of Remark \ref{oss}, we obtain that the Fiber Volume of ${t}_f$ is given by $\vert{t_f}^{-1^*}\frac{Vol_{T^\delta N}}{{t_f}^*Vol_{M \times N}}\vert$ on the image of $t_f$ and it is null otherwise.
		\\In particular, if $\frac{Vol_{T^\delta N}}{{t}^*Vol_{M \times N}}$ is a bounded function, which is \textit{a priori} not clear, then we can conclude that ${t}_f$ is a R.-N.-Lipschitz map. 
		\\Consider the fibered coordinates $\{x^i, \mu^j\}$ on $f^*(TN)$ and the coordinates $\{x^i, y^j\}$ on $M \times N$. Because of the definition of exponential map, the image of ${t}_f$ is contained in a $\delta$-neighborhood of the $Graph(f) \in M \times N$. Then we can cover all the image of ${t}_f$ using the normal coordinates $\{x^i, y^j\}$ around $(p,f(p))$.
		\\Observe that, with respect to these coordinates, we have ${t}_f(0, \mu^j) = (0, \mu^j)$.
		\\Consider $Vol_{f^*T^\delta N}(x, \mu) = \sqrt{det(G_{ij})}(x,\mu) dx^1 \wedge... \wedge d\mu^n$ and $Vol_{M\times N}(x,y) = \sqrt{det(H_{ij})}(x,y)dx^1 \wedge... \wedge dy^n$, where $G_{ij}$ is the matrix of $g_S$ on $f^*(T^\delta N)$ with respect to $\{x^i, \mu^j\}$ and $H_{ij}$ is the matrix of the metric on $M \times N$. Then 
		 \begin{equation}
		 \frac{Vol_{T^\delta N}}{t^*Vol_{M \times N}}(x,\mu) = \frac{\sqrt{det(G_{ij})}}{t_f^*(\sqrt{det(H_{ij})})}(x, \mu) \cdot \frac{1}{det(Jexp_{f(x)}(\mu))}.
		 \end{equation}
		 Observe that in $(0, \mu)$ the matrix $Jexp_{0}(\mu)$ is the identity. Moreover we also have that $G_{ij}(0,\mu)$ is the identity matrix and so $\sqrt{det(G_{ij})}(0,y) = 1$. Finally we obtain
	\begin{equation}
		H_{ij}(0,\mu) = \begin{bmatrix}
			1 && 0 \\
			0 && h_{ij}(\mu)
		\end{bmatrix}
	\end{equation}
	where $h_{ij}$ is the matrix related to the Riemannian metric $h$ in normal coordinates. Then  $det(H_{ij})^{-1}(0,y) \leq C$ because $N$ is a manifold of bounded geometry (Theorem 2.5 of \cite{flow}). This means that
	\begin{equation}
		\frac{Vol_{T^\delta N}}{t^*Vol_{M \times N}}(0,\mu) = \sqrt{\frac{det(G_{ij})}{t_f^*det(H_{ij})}}(0,\mu) \leq C
	\end{equation}
	and so the Fiber Volume of $t_f$ is bounded.
	\end{proof}
	\begin{cor}\label{R.-N.-p}
		Let $f:(M,g) \longrightarrow (N,h)$ be a smooth uniformly proper Lipschitz map between Riemannian manifolds of bounded geometry. Then $p_f$ is a R.-N.-Lipschitz map and $p_f^*$ is $\mathcal{L}^*$-bounded.
	\end{cor}
	\begin{proof}
		 Because of Lemma \ref{svolta}, we know that $t_f$ is a R.-N.-Lipschitz map. Notice that $p_f = pr_N \circ t_f$, where $pr_N: M \times N \longrightarrow N$ is the projection on the second component.
		\\Consider $\overline{f}: f^*T^\delta N \longrightarrow N$ defined as $\overline{f}(w_{f(p)}) := f(p)$,
		Observe that $p_f \sim_\Gamma \overline{f}$. Because of Remark \ref{uniformlyp}, in particular (\ref{eqno}), there is a $C > 0$ such that
		\begin{equation}
		p_f^{-1}(q) \subset A_q := \overline{f}^{-1}(B_C(q)) = \pi^{-1}f^{-1}(B_C(q)),
		\end{equation}
	where $\pi: f^*T^\delta N \longrightarrow M$ is the projection of the bundle.  
		\\This means that if we fix a $q$ in $N$, then the Fiber Volume of $t_f$ in a point $(p,q)$ can be different from zero only if $p \in f^{-1}(B_C(q))$.  
		\\Then, as a consequence of Proposition \ref{compo}, the Fiber Volume of $p_f$ in a point $q$ is given by
		\begin{equation}
		Vol_{p_f}(q) = \int_M Vol_{t_f}(p,q) d\mu_M = \int_{f^{-1}(B_C(q))}  Vol_{t_f}(p,q) d\mu_M \leq K \cdot \mu_M(f^{-1}(B_C(q))).
		\end{equation}
		where $K$ is the supremum of the Fiber Volume of $t_f$. Since $f$ is uniformly proper, then the diameter of $f^{-1}(B_C(q))$ is uniformly bounded and so there is a point $x_0$ in $M$ and a radius $R$ such that
		\begin{equation}
		f^{-1}(B_C(q)) \subseteq B_R(x_0).
		\end{equation}
		Moreover, by Remark (\ref{bvolume}), there is a constant $V$ such that
		\begin{equation}
		\mu_M(f^{-1}(B_C(q))) \leq \mu_M(B_R(x_0)) \leq V
		\end{equation}
		and so
		\begin{equation}
		Vol_{p_f}(q) \leq K \cdot V
		\end{equation}
		 and $p_f$ is a R.-N.-Lipschitz map.
		 \end{proof}
	 \begin{rem}\label{salva}
	 	Consider a smooth map $f:(M,g) \longrightarrow (N,h)$ which is Lipschitz and uniformly proper. Let us suppose, moreover, that $(N,h)$ is a manifold of bounded geometry and $(M,g)$ has bounded Ricci curvature. Then the map $p_f: (f^*(T^\delta N), g_S) \longrightarrow (N,h)$ is well-defined and, moreover, using the same arguments we used in this section, it is also R.-N.-Lipschitz.
	 \end{rem}
	 \subsection{A Thom form for $f^*(TN)$}\label{forma}
	Let us introduce the notion of Thom form.
	 \begin{defn}
	 Let $\pi: E \longrightarrow M$ be a vector bundle. A smooth form $\omega$ in $\Omega_{cv}^*(E)$ is a \textbf{Thom form} if it is closed and its integral along the fibers of $\pi$ is equal to the constant function $1$.
	 \end{defn}
	 Given a Thom form $\omega$ of $f^*TN$ such that $supp(\omega)$ is contained in a $\delta_0 < \delta$ neighborhood of the null section, let us define an operator $e_\omega: \Omega^*(f^*(T^\delta N)) \longrightarrow \Omega^*(f^*(T^\delta N))$, for every smooth form $\alpha$ as
	  \begin{equation}
	  	e_\omega(\alpha) := \alpha \wedge \omega.
	  \end{equation}
	 Our goal, in this subsection, is to find a Thom form $\omega$ such that the operator $e_\omega$ is $\mathcal{L}^*$-bounded. To this end we use the Thom form introduced by Mathai and Quillen in \cite{mathai}. In their work, indeed, they compute a Thom form for a vector bundle endowed with a connection and a metric bundle. In particular, we follow the construction of this form given by Getzler in Proposition 1.3 of \cite{Gez}.
\\
\\First we construct the Thom form of $TN$.
\\Consider the bundle $\pi: (TN, h_S) \longrightarrow (N,h)$ and let $\pi^*(TN)$ be the pullback bundle over $TN$. Consider the bundle metric on $\pi^*(TN)$ given by $\pi^*h$ and fix the connection $\pi^*\nabla^{LC}_h$. Denote by $\Omega^{i,j}$ the algebra
\begin{equation}
	\Omega^{i,j} := \Omega^{i}(TN, \Lambda^j\pi^*TN) = \Gamma(TN, \Lambda^i T^*(TN) \otimes \Lambda^j\pi^*TN).
\end{equation}
We define $\Omega^{*,*} := \bigoplus\limits_{i,j \in \mathbb{N}}\Omega^{i,j}$. This is a bigraded algebra where $\Omega^{i,j}$ are the graded subspaces.
Let us define the section $X:TN \longrightarrow \pi^*TN$ as
\begin{equation}
X(v_p) := (v_p, v_p).
\end{equation}
Fix some normal coordinates $\{U, x^i\}$ on $N$ and let $\{\pi^{-1}(U), x^i, \mu^j\}$ be the coordinates on $TN$ induced by $\{x^i\}$ and by $\{\frac{\partial}{\partial x^i}\}$. Then 
\begin{equation}
X(x, \mu) := \mu^i \frac{\partial}{\partial x^i}.
\end{equation}
Consider the map $\pi^*g(X,X) = \vert X \vert^2: TN \longrightarrow \mathbb{R}$
\begin{equation}
 \vert X \vert^2(v_p) := h_p(v_p, v_p)
\end{equation}
This map can be see as a differential form in $\Omega^{0,0}$. In fibered coordinates it can be expressed as
\begin{equation}
\vert X \vert^2(x, \mu) = h_{ij}(x)\mu^i\mu^j.
\end{equation} 
Consider $\pi^*\nabla^{LC}_h (X)$: this is a form in $\Omega^{1,1}$ and, in local coordinates, it is given by
\begin{equation}
	\begin{split}
\pi^*\nabla^{LC}_h (X) &=  d\mu^i \otimes \frac{\partial}{\partial x^i}  + \mu^i \nabla \frac{\partial}{\partial x^i}\\
&= d\mu^i \otimes \frac{\partial}{\partial x^i} + \mu^i \Gamma^k_{ij}(x) dx^j \otimes \frac{\partial}{\partial x^k} \\
&= (\delta^k_j + \mu^i \Gamma^k_{ij}(x))dx^j \otimes \frac{\partial}{\partial x^k}
\end{split}
\end{equation} 
Finally let us consider $\Omega$ the curvature form of $N$ induced by the Levi-Civita connection $\nabla^{LC}_h$. This is a $2$-form on $N$ with values in $SO(TN)$, which is the bundle of the skew-symmetric endomorphisms of $TM$. Locally it is given by
\begin{equation}
\Omega(x) := R^{j}_{kls}(x) dx^k \wedge dx^l \otimes ( dx^s \otimes \frac{\partial}{\partial x^j}),
\end{equation}
where $R^j_{kls}$ are the components of the Riemann tensor of $N$.
\\Let us identify $\Omega^{i}(N) \otimes SO(TN)$ with $\Omega^{i}(N) \otimes \Gamma(\Lambda^2(TN))$ in the following way: for each $A \in \Omega^{i}(N) \otimes SO(TN)$ we define $\tilde{A}$ as the form locally defined as
\begin{equation}
\tilde{A}(y) := h(e_i(y), A(y)e_j(y))e_i(y) \wedge e_j(y) \in \Lambda^2(T_N)
\end{equation}
where $\{e_i\}$ is an orthonormal frame of $TN$. This means that the curvature $\Omega$, seen as an element of $\Omega^{i}(N) \otimes \Gamma(\Lambda^2(TN))$, is locally given by 
\begin{equation}
\Omega(x) := R^{ij}_{kl}(x) dx^k \wedge dx^l \otimes (\frac{\partial }{\partial x^i} \otimes \frac{\partial}{\partial x^j})
\end{equation}
where $R^{ij}_{kl} = h^{is}R^j_{kls}$. Pulling back $\Omega$ along $\pi$, we obtain $\pi^*\Omega$ which is a differential form in $\Omega^{2,2}$.
\\Let $\phi: \mathbb{R} \longrightarrow \mathbb{R}$ be a smooth function whose support is contained in $[\delta_0, \delta_0]$. Assume that 
\begin{equation}
    (-1)^\frac{n(n+1)}{2}\int_{\mathbb{R}^n} \phi^{(m)}(\frac{\vert x \vert ^2}{2}) dx = 1.
\end{equation}
Then we define
\begin{equation}
\overline{\omega} := \sum_{k=0}^{n} \frac{\phi^{(k)}(\frac{\vert X \vert ^2}{2})}{k!}(\pi^*\nabla X + \pi^*\Omega)^k,
\end{equation}
where $(\pi^*\nabla X + \pi^*\Omega)^k$ is the $k$-times wedge of $\pi^*\nabla X + \pi^*\Omega$.
\\Observe that the support of $\overline{\omega}$ is strictly contained in a $\delta_0$-neighborhood of the zero section. This is a consequence of $supp(\phi) \subseteq [- \delta_0, \delta_0]$.
\\In local fibered coordinate $\{(U,x^i, \mu^j)\}$, the form $\overline{\omega}$  is given by
\begin{equation}
\overline{\omega}(x, \mu) := \alpha_{IJ}^{K}(x, \mu) dx^I\wedge d\mu^J \otimes \frac{\partial}{\partial x^K}
\end{equation}
where, because of $N$ is a manifold of bounded geometry, there is a constant $C$ which does not depend on the choice of $U$, such that $\vert \alpha_{IJ}^{K}(x,\mu)\vert \leq C$.
\\Let us introduce the \textit{Berezin integral} $\textbf{B}$. This is the isometry $\textbf{B}: \Lambda^n(TN) \longrightarrow N \times \mathbb{R}$ defined as
\begin{equation}
\textbf{B}(\alpha_p) := (p, Vol_p(\alpha_p))
\end{equation}
where $n= dim(N)$ and $Vol_p$ is the volume form of $N$ in a point $p$ \footnote{Actually the definition of Berezin integral is much more general: this is the definition of Berezin integral for the fiber bundle $TN$.}. It can be extended to $\mathcal{B}: \Omega^{i,j} \longrightarrow \Omega^i(TN)$ by setting $\mathcal{B}(\alpha \otimes \beta) := \textbf{B}(\beta) \alpha$ if $j = n$, $\textbf{B}(\alpha \otimes \beta) := 0$ otherwise.
\\As showed in Proposition 1.3 of \cite{Gez},
\begin{equation}\label{Thom}
\omega := \mathcal{B}(\overline{\omega})
\end{equation}
is a Thom form. Observe that, in fibered coordinates $\{x,\mu\}$ on $TN$,
\begin{equation}\label{Thombound}
\omega(x, \mu) = \alpha_{IJ}(x, \mu) dx^I\wedge d\mu^J,
\end{equation}
where $\alpha_{IJ}(x,\mu) = det(h_{ij}(x))^{\frac{1}{2}} \alpha_{IJ}^0(x, \mu)$ and $\alpha_{IJ}^0(x, \mu)$ is the coefficient of $dx^I\wedge d\mu^J \otimes (\frac{\partial}{\partial x^1} \wedge ... \wedge \frac{\partial}{\partial x^n})$ in $\overline{\omega}$. So we obtain that $\vert \alpha_{IJ}(x,\mu) \vert $ are uniformly bounded. Moreover, the support of $\omega$ is contained in a $\delta_0$-neighborhood of the null section.
\begin{prop}
Consider $(N,h)$ a manifold of bounded geometry. Let $\alpha$ be a differential form on $(T^\delta N, h_S)$. If in fibered coordinates $\{x^i, \mu^j\}$, where $\mu^j$ refer to $\{\frac{\partial}{\partial x^i}\}$, the coefficients of $\alpha$ are uniformly bounded, then the pointwise norm $\vert\alpha_p\vert$ is uniformly bounded.
\end{prop}
\begin{proof}
It is a direct computation.
\end{proof}
Let $f:(M,g) \longrightarrow (N,h)$ be a smooth Lipschitz map and consider $(f^*(TN), g_{Sf})$ where $g_{Sf}$ is the Sasaki metric induced by $g$, $f^* h$ and $f^*\nabla^{LC}_h$. Observe that if $\omega$ is the Thom form on $TN$ defined in (\ref{Thom}) and, if we consider the map $F: f^*(TN) \longrightarrow TN$ given by $F(p,w_{f(p)}) =(f(p), w_{f(p)})$, then $F^*\omega$ is a Thom form for $f^*(TN)$ and  we also have a uniform bound on the norm of $\vert F^*\omega \vert_p$.
\begin{prop}\label{eomega}
Let us consider a form $\alpha$ on a Riemannian manifold $(M,g)$. Suppose that there is a number $C$ such that $\vert \alpha_p\vert \leq C$ for each $p$ in $(M,g)$. Then the operator $e_{\alpha}(\beta) := \beta \wedge \alpha$ defines an $\mathcal{L}^*$-bounded operator.
\end{prop}
\begin{proof}
It is a direct consequence of Hadamard-Schwartz inequality \cite{Sbordone} which states that, given some linear forms $\alpha_1$, ..., $\alpha_k$ in $\mathbb{R}^n$ with degree $l_1, ..., l_k$, then there is a constant $C_n$ (which only depends on $n$) such that
\begin{equation}
\vert \alpha_1 \wedge ... \wedge \alpha_k\vert \leq C_n \vert \alpha_1 \vert \cdot... \cdot \vert \alpha_k\vert.
\end{equation}
\end{proof}
\begin{rem}\label{Thomgamma}
Consider the bundle $\pi: TN \longrightarrow N$, let $h_N$ be a Riemannian metric of bounded geometry on $N$ and let $\nabla^{LC}_h$ be the Levi-Civita connection. Fix on $TN$ the Sasaki metric $h_S$ induced by $h$ and $\nabla^{LC}_h$. Assume that a group $\Gamma$ acts FUPD on $(N,h)$ by isometries. Observe that on $TN$ there is an action of $\Gamma$ induced by the differential i.e.
\begin{equation}
    \gamma \cdot v_p := d\gamma(v_p).
\end{equation}
Then there is a $\Gamma$-equivariant Thom form $\omega$ for $TN$ which satisfies the assumptions of Proposition \ref{eomega} and its support is contained in a $\delta$-neighborhood of $0_{TN}$. In order to prove this fact, consider the Riemannian covering $s: (N,h) \longrightarrow (N/\Gamma, \tilde{h})$ where $\tilde{h}$ is the Riemannian metric induced by $h$. Then we obtain the map $ds:(TN, h_S) \longrightarrow (T(N/\Gamma), \tilde{h}_S)$. Observe that $TN/\Gamma$ and $T(N/\Gamma)$ are diffeomorphic and $ds$ can be seen as the quotient map. Observe that $h_S$, which is the Sasaki metric on $TN$ induced by $h$, is exactly the pullback metric of $\tilde{h}_S$. Then $ds$ is a local isometry.
\\So $\omega := (ds)^*\alpha$ is a $\Gamma$-equivariant Thom form for $TN$ which satisfies all the assumptions of Proposition \ref{eomega}.
\\
\\Let us consider a smooth Lipschitz map $f:(M,g) \longrightarrow (N,h)$ between manifolds of bounded geometry. Suppose that there is a FUPD action of a group $\Gamma$ by isometries on $M$ and $N$. Assume that $f$ is $\Gamma$-equivariant. There is an action of $\Gamma$ on $f^*TN$ given by
\begin{equation}
\gamma \cdot (p, w_{f(p)}) := (\gamma\cdot p, d\gamma(w_{f(p)})).
\end{equation}
If $\omega$ is a $\Gamma$-equivariant Thom form with support contained in a $\delta$-neighborhood of $0_{TN}$ which satisfies the assumptions of Proposition \ref{eomega}, we already know that $F^*\omega$ is a Thom form with support contained in a $\delta$-neighborhood of $0_{TN}$ which satisfies the assumptions of Proposition \ref{eomega}. Moreover $F^*\omega$ is also $\Gamma$-equivariant. In order to prove this it is sufficient to prove that $F: f^*TN \longrightarrow TN$ is $\Gamma$-equivariant. This is true, indeed
\begin{equation}
F(\gamma p, d\gamma(w_{f(p)})) = d\gamma(w_{f(p)}) = \gamma \cdot w_{f(p)}.
\end{equation}
by definition of $F$.
\end{rem}
	\subsection{The $T_f$ operator}
	Let $(M,g)$ and $(N,h)$ be two manifolds of bounded geometry and let $f:(M,g) \longrightarrow (N,h)$ be a uniformly proper, smooth, Lipschitz map.
	\\Let us denote by $pr_M$ the projection $pr_M: (f^*(T^\delta N),g_{Sf}) \longrightarrow (M,g)$, let $\omega$ be a Thom form of $TN$ defined as in the previous subsection and consider $p_f: (f^*(T^\delta N),g_{Sf}) \longrightarrow (M,g)$ the submersion related to $f$ which we introduced in the previous section. If $f$ is differentiable, then we can define the operator $T_f$ for each smooth $\mathcal{L}^p$-form as
	\begin{equation}
	T_f (\alpha) := pr_{M, \star} \circ e_{F^* \omega} \circ p_f^* (\alpha) =\int_{B^\delta} p_f^*(\alpha) \wedge F^*(\omega) \label{defT}
	\end{equation}
where $B^\delta$ denotes the fibers of $pr_M$. If $f$ is not a smooth Lipschitz map, we consider a smooth Lipschitz map $f'$ which is uniformly homotopic to $f$ \footnote{We know that such a $f'$ exists as a consequence of Proposition \ref{appr}} and we set\footnote{Actually, in this case, the definition of $T_f$ does depend on the choice of $f'$. We will not denote the choice of $f'$ because, as we will see later, $T_f$ induces some operators in (un)-reduced $L^{q,p}$-cohomology and in $L^{q,p}$-quotient cohomology which do not depend on the choice of $f'$. } $T_f := T_{f'}.$
	\begin{prop}
		Let $f$ be a uniform map between Riemannian manifolds of bounded geometry. Then the operator $T_f$ is an $\mathcal{L}^*$-bounded operator.
	\end{prop}
	\begin{proof}
		Let us suppose $f$ is smooth.
		As a consequence of Corollary \ref{R.-N.-p}, $p_f^*$ is $\mathcal{L}^*$-bounded. Because of Corollary \ref{eomega}, we also know that $e_{F^*\omega}$ is an $\mathcal{L}^*$-bounded operator. Finally $pr_{M,\star}$ is $\mathcal{L}^*$-bounded thanks to Proposition \ref{star}. Then $T_f$ is a composition of $\mathcal{L}^*$-bounded operators and so it is $\mathcal{L}^*$-bounded.
	\end{proof}
	\begin{cor}
		Given a uniform map $f:(M,g) \longrightarrow (N,h)$ between two Riemannian manifolds of bounded geometry, then $T_f(dom(d_{min})) \subseteq dom(d_{min})$ and $T_fd = dT_f$. In particular, this means that $T_f$ induces a morphism in $L^{q,p}$-cohomology. Moreover $T_f$ also induces a morphism between the reduced $L^{q,p}$-cohomology groups and the $L^{q,p}$-quotient cohomology groups. 
	\end{cor}
	\begin{proof}
		Let us suppose, again, that $f$ is smooth and Lipschitz.
		In order to prove the Corollary we will prove that the operator $T_f$ satisfies the assumptions of Corollary \ref{boundedness2}. We already know that $T_f$ is $\mathcal{L}^*$-bounded. Then we just have to prove that $T_f(\Omega^*_c(N)) \subseteq \Omega^*_c(M)$ and $dT_f\alpha = T_f d \alpha$ for each $\alpha$ in $\Omega^*_c(N)$.
		\\
		Since $p_f$ is uniformly proper, for each smooth form $\alpha$ in $\mathcal{L}^p(N)$ 
		\begin{equation}
			diam(supp(e_{F^*\omega} \circ p_f^*(\alpha)) \leq C_\alpha.
		\end{equation}
		Because of $supp(pr_{M, \star}\beta) \subset pr_M(supp(\beta))$ for each smooth $\beta$ in $\mathcal{L}^p(f^*T^\delta N)$, we obtain that $supp(T_f \alpha)$ is bounded in $M$. This means that $supp(T_f \alpha)$ is compact since $(M,g)$ is complete and $T_f(\Omega_c(N)) \subseteq \Omega_c(M)$.
		\\
		\\Observe that $e_{F^*\omega} \circ p_f^*(\Omega^*_{c}(N)) \subseteq \Omega^*_{vc}(f^*(T^\delta N))$, where $\Omega^*_{vc}(f^*(T^\delta N))$ is the space of vertically compactly supported smooth forms with respect to the projection $\pi: f^*(T^\delta(M)) \longrightarrow M$. 
		\\By Proposition 6.14.1. of the book of Bott and Tu \cite{bottu}, $\pi_{\star}(\Omega^*_{vc}(f^*T^\delta N) \subseteq \Omega_c(M)$, we obtain
		\begin{equation}
		\int_{B^\delta} d \eta = d \int_{B^\delta} \eta,
		\end{equation}
		if $\eta$ is in $\Omega_{vc}^*(f^*T^\delta N)$. Then we conclude by applying Corollary \ref{boundedness2}.
	\end{proof}
	\begin{rem}
	 Let $f:(M,g) \longrightarrow (N,h)$ be a smooth uniform map between manifolds of bounded geometry. Let $\Gamma$ be a group acting FUPD on $(M,g)$ and $(N,h)$ by isometries and assume $f$ is $\Gamma$-equivariant.
	\\Consider the action of $\Gamma$ on $f^*TN$ defined as $\gamma \cdot (p, w_{f(p)}) := (\gamma \cdot p, d\gamma w_{f(p)})$. Because of Lemma \ref{tilde}, we know that $p_f$ is a $\Gamma$-equivariant map. This means that, if $\alpha$ is a $\Gamma$-equivariant differential form on $N$, then also $p_f^* \alpha$ is $\Gamma$-invariant. Moreover we also know that $F^*\omega$ is $\Gamma$-equivariant (Remark \ref{Thomgamma}). 
	\\Finally  $pr_{M, \star}: \Omega^*_{cv}(f^*TN) \longrightarrow \Omega^*(M)$ preserves the $\Gamma$-equivariance of differential forms. In order to prove this statement consider $X := M/\Gamma$ and $Y:= N /\Gamma$ and let $\tilde{f}: X \longrightarrow Y$ be the map induced by $f$. Denote by $r: M \longrightarrow X$ and $s:N \longrightarrow Y$ the covering maps. Observe that $f^*TN \cong r^*(\tilde{f}^*TY)$: this follows because $TN$ can be identified with $s^*TY$ and so $f^*TN \cong  f^*s^*TY \cong r^*\tilde{f}^*TY$. Let us denote by $pr_{X,\star}$ the integration along the fibers of $\tilde{f}^*TY$ and let $R:r^*\tilde{f}^*TY \longrightarrow \tilde{f}^*TY$ be the bundle map induced by $r$. Then we can conclude by applying Proposition VII of Chapter 5 in \cite{Conn} which allow us to say that
	\begin{equation}
	 pr_{M,\star} \circ R^* = r^* \circ pr_{X,\star}.
	\end{equation}
	Then $T_f = pr_M \circ e_{F^*\omega} \circ p_f^*$ is a $\Gamma$-equivariant operator.
	\end{rem}
\subsection{R.-N.-Lipschitz equivalences of vector bundles}
Let us consider two smooth Lipschitz maps $f_0, f_1:(M,g) \longrightarrow (N,h)$ between Riemannian manifolds and consider a vector bundle $E$ over $N$. Fix a group $\Gamma$ acting on $M$ and on $N$. Then it is a known fact that if $f_0 \sim_\Gamma f_1$ with a Lipschitz homotopy $H$, then
\begin{equation}
	f_0^*(E) \times [0,1] \cong f_1^*(E) \times [0,1] \cong H^*(E) \mbox{   and   } 	f_0^*(E) \cong f_1^*(E)
\end{equation} 
where with $\cong$ we mean that they are isomorphic as vector bundles. As a consequence of this fact we also have that they are homeomorphic as manifolds.
\begin{defn}
Consider two vector bundles $E$, $F$ over a manifold $N$. Let us suppose that $e$ is a Riemannian metric on $E$ and $f$ a Riemannian metric on $F$. Let us denote by $E^\delta$ and $F^\delta$ the $\delta$-neighborhood of the zero sections $0_E$ and $0_F$. The bundles $(E,e)$ and $(F,f)$ are \textbf{R.-N.-Lipschitz equivalent} if there is a bundle isomorphism $\phi: E \longrightarrow F$ such that
\begin{itemize}
    \item $\phi(E^{\delta}) = F^{\delta}$ for each $\delta > 0$,
    \item $\phi_{\vert E^{\delta}}$ and $\phi^{-1}_{\vert F^{\delta}}$ are R.N.-Lipschitz maps.
\end{itemize}
\end{defn}
\begin{rem}\label{ossbundle}
If $E$ and $F$ are R.-N.-Lipschitz equivalent, then if $A: (F^\delta, f) \longrightarrow (M,g)$ is an R.-N.-Lipschitz map, then $\phi^*A = A \circ \phi: (E^\delta, e) \longrightarrow (M,g)$ is again an R.-N.-Lipschitz map. This means that, up to isomorphisms, we can see $A$ as an R.N.-Lipschitz map from $(E^\delta, e)$ to $(M,g)$.
\end{rem}
\begin{prop}\label{bundlelem}
	Let us consider two Riemannian manifolds $(M,g)$ and $(N,l)$, where $(N,l)$ is a manifold of bounded geometry. Consider $f_0, f_1: (M,g) \longrightarrow (N,l)$ two smooth Lipschitz maps and let $H:(M \times [0,1], g + dt^2) \longrightarrow (N,l)$ be a smooth Lipschitz homotopy  between them. Suppose that for each $p$ in $M$ and for each $t_{1}, t_2$ in $[0,1]$
	\begin{equation}\label{bdist}
	 d(H(p,t_{1});H(p,t_2)) \leq \frac{inj_N}{4}.
	\end{equation}
	Let us denote by $g_{S,f_i}$ the Sasaki metric on $f_{i}^*(TN)$ defined by using $g$, $f^*_i(l)$ and $f^*_i(\nabla^{LC}_l)$  where $i = 0,1$.
	\\Then $(f^*_0(TN), g_{Sf_0})$ and $(f^*_1(TN), g_{S,f_1})$ are R.-N.-Lipschitz equivalent.
\end{prop} 
\begin{proof}
We already know that $f^*_0TN$ and $f^*_1TN$ are homeomorphic. In particular, we consider the homeomorphism $h: f^*_1TN \longrightarrow f^*_0TN$ introduced in Proposition 1.7 of the book of Hatcher \cite{AT}. Before to introduce this homeomorphism we need a specific cover $\{U_i\}$ for $M \times [0,1]$ and some specific orthonormal frames of $H^*TN$ on $U_i$. In his proof, indeed, Hatcher uses a generic cover $\{U_i\}$ and some generic frames of $H^*TN$, but, in our setting, we need $U_i$ and some frames which satisfy some further conditions.
\\Let us start by the cover.
\\Fix two $N$-small numbers $\delta_1$ and $\delta_2$ such that $\delta_2 \leq \frac{inj_N}{2}$ and $\delta_1 \leq \delta_2 - \frac{inj_N}{4}$. By Lemma 2.16 of \cite{bound} we know that there is a cover $\{B_{\delta_1}(q_i)\}$ of $N$ such that a ball $B_{\delta_2}(x)$ in $N$ intersects at most $K$ balls $\{B_{\delta_2}(q_i)\}$. Consider the cover of $M \times [0,1]$ given by $\{H^{-1}(B_{\delta_1}(q_i))\}$. Observe that
\begin{equation}\label{inclu}
H^{-1}(B_{\delta_1}(q_i)) \subseteq \pi_0^{-1}(\pi_0(H^{-1}(B_{\delta_1}(q_i)))) \subseteq H^{-1}(B_{\delta_2}(q_i))
\end{equation}
where $\pi_0: M \times [0,1] \longrightarrow M$ is the projection on the first component.
The first inclusion of (\ref{inclu}) is obvious. The second one is a consequence of (\ref{bdist}), indeed if $p$ is in $\pi(H^{-1}(B_{\delta_1}(q_i)))$ then there is a $t_1$ such that $H(p,t_1)$ is in $B_{\delta_1}(q_i)$. Then for all $t_2$ in $[0,1]$ the inequality (\ref{bdist}) holds. This implies that
\begin{equation}
d(H(p,t_2); q) \leq d(H(p,t_{1});H(p,t_2)) + d(H(p,t_{1}), q) \leq \frac{inj_N}{4} + \delta_1 \leq \delta_2
\end{equation}
and so $(p,t_2)$ is in $H^{-1}(B_{\delta_2}(q_i))$ for each $t_2$ in $[0,1]$. Let us define
\begin{equation}
U_i := \pi_0(H^{-1}(B_{\delta_1}(q_i))).
\end{equation}
Let us consider the cover $\{U_i \times [0,1]\}$. Observe that $H(U_i \times [0,1]) \subseteq B_{\frac{inj_N}{2}}(q_i)$ for each $i$. Moreover there is a number $K$ such that for each $i$ the intersection $(U_i \times [0,1]) \cap (U_j \times [0,1])$ is not empty.
\\Let us introduce our orthonormal frame of $TN$ on $B_{\frac{inj_N}{2}}(q_i)$.
\\Consider for each $i$ some normal coordinates $\{y^j\}$ around $q_i$. Fix on $T_{q_i}N$ the orthonormal basis $\{s_j(q_i)\}$ where $s_j(q_i) := \frac{\partial}{\partial y^j}(q_i)$. Using parallel transport along geodesics emanating from $q_i$, we obtain a frame $\{s_r\}$. Let us denote by $\{\theta^r\}$ the dual frame of $\{s_r\}$ and let $a^r_j,b^r_j: B_{\frac{inj_N}{2}}(q_i) \longrightarrow \mathbb{R}$ be the function defined by
\begin{equation}
    \theta^s(y) := a^s_j(y)dy^j \mbox{     and    }  dy^s := b^s_j(y) \theta^j.
\end{equation}
As shown on pages 6-9 of \cite{flow}, for each $k$ in $\mathbb{N}$, the partial derivatives of order $k$ of $\vert a^i_j(y) \vert$ and of $\vert b^i_j(y) \vert$ are uniformly bounded by a constant $C_k$ which does not depends on $B_{\frac{inj_N}{2}}(q_i)$ or on $\{y^j\}$. Consider the Riemannian metric $l$ on $N$. Let us define $A^s_j := l^{sr}a_r^ml_{mj}$ and $B^s_j := l_{sr}b^r_ml^{mj}$. Then
\begin{equation}
    s_r(y) = A^w_r(y)\frac{\partial}{\partial y^w}(y) \mbox{     and    }  \frac{\partial}{\partial y^w}(y) := B^r_w(y)  s_r(y).
\end{equation}
Observe that also the partial derivatives of order $k$ of $\vert A^i_j(y) \vert$ and of $\vert B^i_j(y) \vert$ are uniformly bounded.
\\The frame $\{s_j\}$ can be used to identify $\pi: H^*TN \longrightarrow M \times [0,1]$ on $U_i \times [0,1]$ and $U_i \times [0,1] \times \mathbb{R}^n$. Indeed we can identify $(p, t, \mu^j s_j(H(p,t)))$ in $\pi^{-1}(U \times [0,1])$ with $(p, t, \mu^1, ..., \mu^n)$ in $U \times [0,1] \times \mathbb{R}^n$.
\\
\\Let us introduce the isomorphism $h$. By Lemma 2.17 of \cite{bound}, there is a partition of unity $\{\xi\}$ of $N$ referred to the cover $\{B_{\delta_1}(q_i)\}$ where each $\xi$ is a smooth Lipschitz function. Consider $\{\phi_i\}$ where $\phi_i := f^{*}_0(\xi)$. We know that $f_0^{-1}(B_{\delta_1}(q_i)) \subseteq \pi_0(H^{-1}(B_{\delta_1}(q_i)) = U_i$. This means that $\{\phi_i\}$ is a smooth Lipschitz partition of unity subordinate to $U_i$. Let $\psi_i: M \longrightarrow [0,1]$ be 
\begin{equation}
\psi_i := \phi_1 + \phi_2 + ... + \phi_i.
\end{equation}
Let $M_i \subset M \times [0,1]$ be the graph of $\psi_i$ and let $\pi_i: E_i \longrightarrow M_i$ be the restriction of $H^*TN$ to $M_i$. Since $H^*TN$ is trivial on $U_i \times [0,1]$, then homeomorphism $e_i: M_i \longrightarrow M_{i+1}$ defined as $e_i(p, \psi_i(p)) := (p, \psi_{i-1}(p))$ lifts to a homeomorphism $h_i: E_i \longrightarrow E_{i-1}$. This homeomorphism is the identity outside $\pi^{-1}(U_i \times [0,1])$ and on $p^{-1}(U_i \times [0,1])$ it is defined by
\begin{equation}
h_i(p, \psi_i(p), v) := (p, \psi_{i-1}(p), v).
\end{equation}
Then homeomorphism $h$ is given by the composition
\begin{equation}
h :=  h_1 \circ h_2 \circ h_3 \circ ....
\end{equation}
We already know from Proposition 1.7 of \cite{AT} that $h$ is an isomorphism of vector bundles and so, in particular, a diffeomorphism. \\Fix on $M \times [0,1]$ the metric $g + dt^2$. Let us consider the maps $\tilde{f}_i: M \longrightarrow N$ defined as $\tilde{f}_i(p) := H(p, \psi_i(p))$ and consider the Sasaki metric $g_{S,\tilde{f}_i}$ on $E_i$ induced by the metric $g$ on $M$, the pullback metric $\tilde{f}_i^*l$ and the pullback connection $\tilde{f}_i^*\nabla^{Lc}_l$.
Observe that outside $\pi^{-1}(U_i \times [0,1])$ the map $h_i: (\tilde{f}^*_{i}TN, g_{S,f_i}) \longrightarrow (\tilde{f}^*_{i-1}TN, g_{S,f_{i-1}})$ is an isometry. This follows because $\tilde{f}_{i-1} = \tilde{f}_i$ outside $\pi^{-1}(U_i \times [0,1])$. Then the norm of $dh_{i}$ and the Fiber Volume of $h_i$ are both equal to $1$ outside $\pi^{-1}(U_i \times [0,1])$. Let us consider $h_i$ on $\pi^{-1}(U_i \times [0,1])$.
\\
\\Consider $p$ in $U_i$. Recall that we have on $B_{\frac{inj_N}{2}}(q_{i+1})$ some normal coordinates $\{y^z\}$. Fix a constant $K$ such that $ K > \sup \{1, \vert \vert dH_(q,t) \vert \vert\}$ for each $(q,t)$ in $M \times [0,1]$. Choose some normal\footnote{With respect to $g$} coordinates $\{W, x^j\}$ around $p$ such that in $W \times [0,1]$ the norms of the derivatives of $H: W \times [0,1] \longrightarrow B_{\frac{inj_N}{2}}(q_{i+1})$ and the norm of the Gram matrix of $g$ with respect to $\{x^i\}$ are uniformly bounded by $K$. Let $\{x^j, \mu^z\}$ be the fibered coordinates on $\tilde{f}_{i+1}^*(TN)_{\vert U_{i+1}}$ related to the frame $\{\frac{\partial }{\partial y^z}\}$. On the other hand we can also define the fibered coordinates $\{x^j, \nu^z\}$ on $\tilde{f}_{i+1}^*(TN)_{\vert U_{i+1}}$ related to the frame $\{s_j\}$.
\\In the same way we obtain $\{x^j, \sigma^z\}$ which are the fibered coordinates on $\tilde{f}_{i}^*(TN)$ related to the frame $\{\frac{\partial }{\partial y^j}\}$ and the coordinates $\{x^j, \tau^z\}$ which are the fibered coordinates related to the frame $\{s_j\}$. 
\\With respect to the coordinates related to the frame $\{s_j\}$ we have $h_i(x, \nu) = (x, \nu)$, while, with respect to the coordinates related to $\{\frac{\partial }{\partial y^j}\}$, we obtain
\begin{equation}
h_{i+1}(x, \mu) = (x^j, A_r^j(\tilde{f}_{i}(x))B_w^r(\tilde{f}_{i+1}(x))\mu^w).
\end{equation}
The norms of $\frac{\partial }{\partial x^i}$ and $\frac{\partial }{\partial \mu_j^m}$ with respect to $g_{S,\tilde{f}_{i+1}}$ and the norms of $\frac{\partial }{\partial x^i}$ and $\frac{\partial }{\partial \sigma_j^m}$ with respect to $g_{S,\tilde{f}_{i}}$ are uniformly bounded by a constant $L$ which does not depend on $i$ (this is a consequence of (\ref{metri})). Then, since the norms of the derivatives of $A_r^j$ and $B_w^r$ are uniformly bounded by a constant $C$, we obtain that the Lipschitz constant of $h_i$ is less or equal to $C^2\cdot L$.
\\Let us focus on the Fiber Volume of $h_i$ on $U_i$.
\\Observe that, since $h_i(x, \nu) = (x, \nu)$,
\begin{equation}
\begin{split}
Vol_{h_{i+1}}(x, \nu) &= h_{i+1}^* \sqrt{\frac{det(G_{rj})}{det(L_{jk})}}(x, \nu) \\
&=\sqrt{det(G_{rj})det(L_{jk})^{-1}}(x, \nu).
\end{split}
\end{equation}
where $G_{rj}$ is the matrix related to the metric $g_{S,f^*_{i+1}}$ with respect to the coordinates $\{x, \nu \}$ and $L_{jk}$ is the matrix related to $g_{S,f^*_{i}}$ with respect to the coordinates $\{x, \tau \}$. Observe that
\begin{equation}
det(G_{rj})(x, \nu) = det(\hat{G}_{rj}))(A(x, \nu)) \cdot det(J_A(x, \nu))^2,
\end{equation}
where $A$ is the change of coordinates from $(x, \nu)$ to $(x, \mu)$, $J_A$ is its Jacobian and $\hat{G}_{rj}$ is the matrix related to the metric $g_{S,\tilde{f}^*_{i+1}}$ with respect to the coordinates $\{x, \mu \}$. Define
\begin{equation}
A(x, \nu) := (x^l, A^r_j(f_{i}(x))\nu^j).
\end{equation}
Notice that the determinant of $\hat{G}_{kr}$ is uniformly bounded on each $\delta$-neighborhood of the $0$-section of $f^*_{i+1}TN$ (this is a consequence of Formula (\ref{metri})). Then
\begin{equation}
\vert det(G_{rj})(x, \nu) \vert = \vert det(\hat{G}_{kr}))(A(x, \nu)) \vert \cdot \vert det(J_A(x, \nu)) \vert ^2 \leq J(d_{\tilde{f}_{i+1}TN}((x, \mu), 0_{f^*_{i+1}TN})
\end{equation}
where $J: \mathbb{R} \longrightarrow  \mathbb{R} $ is a function which does not depend on $i$.
\\Let us focus on $det(L_{jk})^{-1}$. We have
\begin{equation}
det(L_{rj})^{-1}(x, \tau) = det(\hat{L}_{kr}))^{-1} (\overline{A}(x, \tau)) \cdot det(J_{\overline{B}}(x, \tau))^2,
\end{equation}
where $\overline{A}$ is the change of coordinates from $(x, \tau)$ to $(x, \sigma)$, $J_{\overline{B}}$ is the Jacobian of its inverse and $\hat{L}_{kr}$ is the matrix related to the metric $g_{S,f^*_{i}}$ with respect to the coordinates $\{x, \sigma\}$. Let us define
\begin{equation}
A(x, \nu) := (x^l, A^r_j(\tilde{f}_{i}(x))\tau^j).
\end{equation}
and 
\begin{equation}
A(x, \nu)^{-1} := (x^l, B^r_j(\tilde{f}_{i}(x))\sigma^j).
\end{equation}
Notice that the determinant of $\hat{L}_{kr}^{-1}$ is uniformly bounded on each $\delta$-neighborhood of the $0$-section of $f^*_i(TN)$ (this is a consequence of Formula (\ref{invmetri})). Then
\begin{equation}
\vert det(L_{ij})^{-1}(x, \tau) \vert = \vert det(\hat{L}_{kr})^{-1} (\overline{A}(x, \tau))\vert \cdot \vert det(J_{\overline{B}}(x, \tau))\vert ^2 \leq F(d_{\tilde{f}_i}((x, \mu), 0_{f^*_{i+1}}),
\end{equation}
where $F: \mathbb{R} \longrightarrow  \mathbb{R} $ is a function which does not depend on $i$.
This means that for each $\delta$ the Fiber Volume of $h_{i \vert_{\pi_i^{-1}(U_i)}}$ is bounded in each $\delta$-neighborhood of the $0$-section of $f^*_{i}TN$ by a constant $C(\delta)$.
	\\Then, in order to conclude the proof, we just have to observe that, since for each $p$ in $M$ there are at most $K$ of $U_i$'s such that $p \in U_i$, then the Lipschitz constant of $h$ is bounded by $C^{2K}\cdot L^K$ and the Fiber Volume of $h$ is uniformly bounded by $C(\delta)^K$. The same also happens for $h^{-1}$ and so $f^*_0(TN)$ and $f^*_1(TN)$ are R.N.-Lipschitz equivalent.
\end{proof}
\begin{prop}\label{equiv}
	Let us consider $f_0, f_1: (M,g) \longrightarrow (N,h)$ two smooth Lipschitz maps such that $f_0$ and $f_1$ are Lipschitz-homotopic. Then $(f^*_0(TN), g_{S,0})$ and $(f^*_1(TN), g_{S,1})$ are R.-N.-Lipschitz equivalent.
\end{prop}
\begin{proof}
	Let us consider a smooth Lipschitz homotopy $H$ between $f_0$ and $f_1$. Consider a finite partition  $\{[s_i, s_{i+1}]\}$ of $[0,1]$ such that $s_{i+1} - s_i \leq \frac{inj_N}{4C_{H}}$, where $C_H$ is the Lipschitz constant of $H$. Let us define the maps $H_{s}: (M,g) \longrightarrow (N,h)$ as $H_i(p) := H(p,s_i)$ then we can observe that
	\begin{equation}
	d(H_{i}(p), H_{i+1}(p)) \leq d(H(p, s_i),(p, s_{i+1})) \leq C_H\cdot(s_{i+1} - s_i) = \frac{inj_N}{4}.
	\end{equation}
Then it is sufficient to apply the previous Lemma to $H_i$ and $H_{i+1}$ and observe that a composition of R.-N.-Lipschitz map is a R.-N.-Lipschitz map.
\end{proof}
\begin{cor}\label{equivhomo}
Consider a Lipschitz map $H: (M \times [0,1], g + dt^2) \longrightarrow (N,h)$. Suppose that $g$ satisfies the same assumption as Proposition \ref{bundlelem}. For each $i$ in $[0,1]$, denote by $f_i:(M,g) \longrightarrow (N,h)$ the map defined as $f_i(p) := H(p,i)$. Assume that there is an $\varepsilon$-neighborhood of $i$ in $[0,1]$ such that $H(x, t) = f_i(x)$ on this neighborhood. Then the vector bundles $(H^*(TN), g_{S,H})$ and $(f^*_i(TN) \times [0,1], g_{Sf_i} + dt^2)$ are R.-N.-Lipschitz equivalent.
\end{cor}
\begin{proof}
Let us define the map $\overline{f}_i: M \times [0,1] \longrightarrow N$ as $\overline{f}_{i}(p,s) := f_i(p)$. Observe that the Sasaki metric on $\overline{f}_i^*(TN)$ defined by using $g$, $\overline{f}^*_i(h)$ and $\overline{f}^*_i(\nabla^{LC}_h)$ is the product metric between the Sasaki metric on $f^*_i(TN)$ and the metric $dt^2$ on $[0,1]$.
\\Then we can conclude by observing that the map
\begin{equation}
\mathcal{H}: (M \times [0,1] \times [0,1], g + ds^2 + dt^2) \longrightarrow (N,h)
\end{equation}
defined as $\mathcal{H}(p,s,t) := H(p, i \cdot t + s \cdot (1-t))$ is a Lipschitz-homotopy between $H$ and $\overline{f}_i$.
\end{proof}
\begin{cor}
Consider two smooth Lipschitz maps $f_0, f_1: (M,g) \longrightarrow (N,h)$, assume that $g$ satisfies the same assumptions of Lemma \ref{bundlelem}. Suppose that $f_0$ and $f_1$ are Lipschitz-homotopic with a smooth Lipschitz homotopy $H$. Assume that there are two $\varepsilon$-neighborhoods $U_0$ and $U_1$ of $M\times \{0\}$ and $M \times \{1\}$ such that $H(p,t) = f_i(p)$ for each $(p,t)$ in $U_i$. Let us denote by 
\begin{equation}
\Phi: (f_0^*(TN) \times [0,1], g_{S,f_0} + dt^2) \longrightarrow (H^*TN, g_{SH})
\end{equation}
the R.N.-Lipschitz equivalence of vector bundles. Then the restriction 
\begin{equation}
\phi := \Phi_{\vert_{f_0^*(TN) \times \{0\}}}: (f_0^*(TN), g_{S,f_0}) \longrightarrow (f_1^*(TN), g_{S,f_1}).
\end{equation}
is a R.N.-Lipschitz equivalence of vector bundles.
\end{cor}
\begin{proof}
We know that $\phi$ is a vector bundle isomorphism: its inverse is given by the inverse of $\Phi$ restricted to the bundle $H^*TN_{\vert M \times \{1\}}$. Moreover the injections $j_i: (f_i^*(TN), g_{Sf_i}) \longrightarrow (H^*TN, g_H)$ are isometric embedding since $H$ is constant around $M \times \{i\}$. Then the bounds on the Fiber Volume and on the Lipschitz constant of $\phi$ are the same of $\Phi$.
\end{proof}
\begin{rem}\label{equalizer}
Let $i= 0,1$ and consider two smooth Lipschitz maps $f_i: (M,g) \longrightarrow (N,h)$. Let us suppose that $h$ is a Lipschitz homotopy between $f_0$ and $f_1$. Let $p_h$ and $p_{f_i}$ be the submersion defined in Lemma \ref{tilde} related to $h$ and to $f_i$. Then 
\begin{equation}
    p_h \circ \Phi: (f_0^*(TN) \times [0,1], g_{S,f_0} + dt^2) \longrightarrow (N,h)
\end{equation}
is a smooth Lipschitz homotopy between $p_{f_0}$ and 
\begin{equation}
p_{f_1} \circ \phi: (f_0^*(T^\delta N), g_{S,f_0}) \longrightarrow (f_1^*(T^\delta N), g_{S,f_1}).
\end{equation}
Moreover it is also true that
\begin{equation}
\begin{split}
    T_{f_1} &= pr_{M1\star} \circ e_{F_1^*\omega} \circ p_{f_1}^* \\
    &= pr_{M1\star} \circ (\phi^{*})^{-1}  \circ \phi^* \circ e_{F_1^*\omega} \circ (\phi^{*})^{-1}  \circ \phi^* \circ p_{f_1}^* \\
    &= pr_{M0\star} \circ e_{\phi^* F_1^*\omega} \circ (p_{f_1} \circ \phi)^*,
\end{split}
\end{equation}
where $pr_{Mi}: f^*_i(TN) \longrightarrow M$ is the projection of the bundle. This means that up to R.N.-Lipschitz isomorphisms, we can consider $p_{f_1}$ as a map defined on $f_0^*TN$ and $p_h$ as a homotopy between $p_{f_0}$ and $p_{f_1}$. From now on we will not write the isomorphisms $\Phi$ or $\phi$ every time.
\end{rem}
	\subsection{Lemmas about homotopy}\label{lemhom}
	In this section we study the $\mathcal{L}^*$-boundedness of the pullback of some homotopies.
	\begin{lem}\label{homo1}
		Let $f: (M,g) \longrightarrow (N,h)$ be a smooth Lipschitz map and let $\delta \leq inj_N$. Consider the homotopy $H: (f^*(T^\delta N) \times [0,1], g_S + dt^2) \longrightarrow (N,l)$ defined as
		\begin{equation}
			H(p,w,t) := p_f(t\cdot w_{f(p)}).
		\end{equation}
		Then, if $f$ is a smooth R.-N.-Lipschitz map, then also $H$ and $(H, id_{[0,1]}): (f^*(T^\delta N) \times [0,1], g_S + dt^2) \longrightarrow (N \times [0,1], l + dt^2)$ are R.-N.-Lipschitz maps.
	\end{lem}
	\begin{proof}
		Observe that $H = pr_{N} \circ (H, id_{[0,1]})$, where $pr_N: N \times [0,1] \longrightarrow N$ is the projection of the first component. In particular $pr_N$ is a R.-N.-Lipschitz map. This means that if we prove that $(H, id_{[0,1]})$ is a R.-N.-Lipschitz map, then also $H$ is a R.-N.-Lipschitz map.
		\\Observe that $f^*(T^\delta N) \times \{0,1\}$ is a set of null measure and so we can consider the interval $[0,1]$ as open.
		\\Let us define the map 
		\begin{equation}
			(pr_M, H, id_{(0,1)}): (f^*T^\delta N \times (0,1), g_S + dt^2) \longrightarrow (M \times N \times (0,1), g \times h + dt^2)
		\end{equation}
		given by $(pr_M, H, id_{(0,1)})(v_{f(p)}, s) := (p, p_f(s \cdot v_p), s)$. This map is a submersion since we are considering $s \neq 0$.
	    In particular $(pr_M, H, id_{(0,1)})$ is Lipschitz since it is a composition of Lipschitz maps.
	    \\Moreover $(H,id_{(0,1)}) = pr_{N \times (0,1)} \circ (pr_M, H, id_{(0,1)})$ and so, by applying Proposition \ref{compo}, we have
	    \begin{equation}
	    Vol_{(H,id_{(0,1)})}(q,t) = \int_{M} Vol_{(pr_M, H, id_{(0,1)})}(p, q, t) d\mu_M.
	    \end{equation}
		Let us calculate the Fiber Volume of $(pr_M, H, id_{(0,1)})$.
		\\Similarly to the Lemma \ref{svolta}, consider for each $p$ in $M$ some normal coordinates $\{x^i\}$ around $p$, some normal coordinates $\{y^j\}$ around $f(p)$ in $N$. Fix the frame $\{\frac{\partial}{\partial x^i}\}$, we obtain the coordinates $\{x^i, \mu^j, t\}$ on $f^*(TN)\times [0,1]$ around $(p, 0, 0)$.
		\\Moreover we can also fix the coordinates $\{x^i, y^j, t\}$ on $M\times N \times [0,1]$. \\Observe that the coordinates $\{x^i, y^j, t\}$ are enough to cover the image of $(id_M, H, id_{[0,1]})$ which is contained in a $\delta$-neighborhood of $Graph(f) \times [0,1]$.
		\\Finally, with respect to these coordinates, we obtain
		\begin{equation}
			(pr_M, H, id_{(0,1)})(0, \mu^j, t) = (0, t \cdot y^j, t).
		\end{equation}
		Similarly as we did in Lemma \ref{svolta}, the Volume forms $Vol_{f^*(TN) \times [0,1]}(0,\mu,t) = dx^1 \wedge ... \wedge dx^m \wedge d\mu^1 \wedge ... \wedge d\mu^n \wedge dt$ and $Vol_{M \times N \times [0,1]}(0, y, t) = \sqrt{det(H_{ij}(0,y))}dx^1 \wedge ... \wedge dx^m \wedge dy^1 \wedge ... \wedge dy^n \wedge dt$, where
		\begin{equation}
			H_{ij}(0,y) = \begin{bmatrix}
				1 && 0 && 0 \\
				0 && h_{l,s}(y) && 0 \\
				0 && 0 && 1
			\end{bmatrix}
		\end{equation}
		and $h_{l,s}$ are the components of the metric $h$ on $N$. Then $(pr_M, H, id_{[0,1]})$ is a diffeomorphism with its image and so, by applying Remark \ref{oss}, its Fiber Volume on $im(pr_M, H, id_{[0,1]})$ is given by
		\begin{equation}
			\vert[(pr_M, H, id_{[0,1]})^{-1}]^* \frac{Vol_{f^*TN \times [0,1]}}{(pr_M, \tilde{h}, id_{[0,1]})^*Vol_{M \times N \times [0,1]}}\vert .
		\end{equation}
		Observe that, similarly to Lemma \ref{svolta}, in a point $(0, y^j, t)$ the Fiber Volume of $(pr_M, \tilde{h}, id_{[0,1]})$ is
		\begin{equation}
			\frac{Vol_{f^*TN \times [0,1]}}{(pr_M, H, id_{[0,1]})^*Vol_{M \times N \times [0,1]}} = \frac{1}{t^n} (1 + C(t, y)),
		\end{equation}
		where $C$ is a bounded function.
		\\Consider the projection $pr_{N \times [0,1]}: M \times N \times (0,1) \longrightarrow N \times (0,1)$. Let us recall that, thanks to the Proposition \ref{compo}, we have
		\begin{equation}
			Vol_{(H,id_{(0,1)})}(q,t) = \int_M Vol_{(id_M, H, id_{[0,1]})}(p,q,t) d\mu_M.
		\end{equation}
		We know that outside the image of $(id_M, H, id_{[0,1]})$ the Fiber Volume $Vol_{(id_M, H, id_{[0,1]})}$ is null. Let us denote by $H_t: M \times N \longrightarrow N$ the map defined as $H_t(p,q) := H(p,q,t)$. Observe that $Vol_{(id_M, H, id_{[0,1]})}$ on $M \times \{q\} \times \{t\}$ is null outside
		 \begin{equation}
		 pr_M(H^{-1}_t(q)) \times \{q\} \times \{t\} \supseteq [im(id_M, H, id_{[0,1]})] \cap [M \times \{q\} \times \{t\}].
		 \end{equation}
		 Since $H$ is Lipschitz, by applying Remark \ref{uniformlyp}, we have
		\begin{equation}
			H^{-1}_t(q) \subseteq \pi^{-1}(f^{-1}(B_{C_H\cdot t}(q))) \times \{t\}
		\end{equation}
	where $\pi: f^*TN \longrightarrow M$ is the projection of the bundle and $C_H$ is the Lipschitz constant of $H$. So
		\begin{equation}
			pr_M(H^{-1}_t(q)) \times \{q\} \times \{t\} \subset f^{-1}(B_{C_H\cdot t}(q))\times \{q\} \times \{t\}.
		\end{equation}
		Observe that, since $N$ is a manifold of bounded geometry,
		\begin{equation}
			\mu_N(B_{C_H \cdot t}(q)) \leq C_H^nt^n(1 + L(t))
		\end{equation}
		where $L$ is a bounded function which depends by $t$.
		\\Then because of the fact that $f$ is a R.-N.-Lipschitz map, we obtain
		\begin{equation}
			\mu(f^{-1}(B_{C_H\cdot t}(q))) \leq K_0 \mu_N(B_{C_H \cdot t}(q)) \leq K_0C_H^nt^n(1 + L(t)),
		\end{equation}
		 Then  the Fiber Volume of $(H, id_{(0,1)})$ is given by
		\begin{equation}
			\begin{split}
				Vol_{(H, id_{(0,1)})}(q,t) &= \int_{f^{-1}(B_{C_H\cdot t}(q))} Vol_{(id_M, H, id_{[0,1]})}(p,q,t) d\mu_M \\
				&\leq \frac{1}{t^n}(1 + C(t, \tilde{y})) \cdot (K_0C_H^nt^n(1 + L(t))) \\
				&\leq K \cdot \frac{1}{t^n} \cdot t^n \leq K.
			\end{split}
		\end{equation}
		and so $(H, id_{(0,1)})$ is a R.-N.-Lipschitz map and so also
		\begin{equation}
			H = pr_N \circ (H, id_{(0,1)})
		\end{equation}
		is a R.-N.-Lipschitz map.
	\end{proof}
Let us consider two smooth uniformly proper Lipschitz maps $g: (S,v) \longrightarrow (M,m)$ and $f:(M,m) \longrightarrow (N,l)$ between manifolds of bounded geometry. Let us define the map $\overline{g}: g^*T^\sigma M \longrightarrow M$ defined as $\overline{g}(v_{g(p)}) := g(p)$ and consider the submersion $p_g: g^*T^\sigma M \longrightarrow M$ related to $g$.
\\Consider, moreover, the compositions $f \circ \overline{g}$ and $f \circ p_g:  g^*T^\sigma M \longrightarrow N$. Observe that these maps are Lipschitz-homotopic. Then, because of Remark \ref{equalizer}, the pullback bundles over $S$ given by $(f \circ \overline{g})^*TN$ and $(f \circ p_g)^*TN$ are R.-N.-Lipschitz homotopy equivalent. This means that, up to R.N.-Lipschitz equivalence of vector bundles, we can suppose 
\begin{equation}
    (f \circ \overline{g})^*TN = (f \circ p_g)^*TN.
\end{equation}
Then we can see the maps $p_{f \circ p_g}$ and $p_{f \circ \overline{g}}$ as defined on the same domain $(f \circ \overline{g})^*TN$.
\\Denote by $h: g^*(T^\sigma M) \times [0,1] \longrightarrow N$ the homotopy defined as $h(v_{g(p)}, t) := f \circ p_g(t \cdot v_{g(p)})$ between $f \circ p_g$ and $f \circ \overline{g}$. Because of Corollary \ref{equivhomo} the bundles $(h^*(TN), g_{S,h})$ and $(f \circ \overline{g}^*(T^\delta N) \times [0,1], g_{S} + dt^2)$ are R.-N.-Lipschitz equivalent. In particular, according to Remark \ref{equalizer}, we consider
\begin{equation}
p_h: ((f \circ \overline{g})^*(T^\delta N) \times [0,1], g_{S} + dt^2) \longrightarrow (N,l)
\end{equation}
is a R.-N.-Lipschitz map because $h$ is smooth, uniformly proper and Lipschitz map and because of Remark \ref{salva}.
\\Observe that $p_h$ is a Lipschitz homotopy between the maps $p_{f \circ \overline{g}}$ and $p_{f \circ p_{g}}: f \circ \overline{g}^*(T^\delta N) \longrightarrow (N,l)$.
\\Applying the point 4. of Proposition \ref{tilde}, we obtain that
\begin{equation}\label{cosse}
		p_{f \circ p_{g}} = p_{id_M} \circ F \circ P_g = p_f \circ P_g
\end{equation}
where $F: f^*(T^\delta N) \longrightarrow T^\delta N$ is the bundle map induced by $f$ and 
\begin{equation}
P_g: (f \circ \overline{g})^*T^\delta N = (f \circ p_g)^*T^\delta N \longrightarrow f^*T^\delta N
\end{equation}
is the bundle map induced by $p_g$. Then we obtain the following Proposition.
\begin{prop}\label{homo2}
The map $H: ((f \circ \overline{g})^*T^\delta N \times [0,1], g_{S} + dt^2) \longrightarrow (N, l)$ defined as
\begin{equation}
H(v_{f \circ \overline{g}},t) := p_f(P_g(t\cdot v_{f \circ \overline{g}})) = p_h
\end{equation}
is a R.-N.-Lipschitz map and it is the homotopy between $p_{f \circ p_{g}}$ and $p_{f \circ \overline{g}}$.
\end{prop}
	\subsection{Uniform homotopy invariance of (un)-reduced $L^{q,p}$-cohomology}
	Before the final proof we need some Lemmas.
	\begin{lem}\label{GLee}
		Let $(M,g)$ be a Riemannian manifold and consider $([0,1], dt^2)$. Then there is an $\mathcal{L}^*$-bounded operator $\int_{0 \mathcal{L}}^1: \mathcal{L}^p(M\times[0,1]) \longrightarrow \mathcal{L}^p(M)$ such that for all smooth $\alpha \in \mathcal{L}^p(M\times[0,1])$,
		\begin{equation}\label{taglia}
		i_1^*\alpha - i_0^*\alpha = \int_{0 \mathcal{L}}^1d \alpha + d\int_{0 \mathcal{L}}^1 \alpha
		\end{equation}
		Moreover $\int_{0 \mathcal{L}}^1$ sends compactly-supported differential forms on $\Omega^*(M \times [0,1])$ to $\Omega^*_c(M)$.
	\end{lem}
	\begin{proof}
		Let $\alpha$ be in $\Omega^*_c(M\times [0,1])$ with compact support and let $p:M\times[0,1] \longrightarrow M$ be the projection on the first component. Then we can decompose every differential form on $M \times [0,1]$ as a sum of forms
		\begin{equation}
		\alpha = g(x,t)p^*\omega +  f(x,t)dt \wedge p^*\omega,
		\end{equation}
		for some $\omega$ in $\Omega_c(M)$ and for some $C^{\infty}$-class functions $g,f:M\times[0,1] \longrightarrow \mathbb{C}$. Then we can define the linear operator $\int_{0,\mathcal{L}}^{1}$ as follow: if $\alpha$ is a $0$-form with respect to $[0,1]$, then
		\begin{equation}
		\int_{0,\mathcal{L}}^{1}\alpha := \int_{0,\mathcal{L}}^{1} g(x,t)p^*\omega =  0
		\end{equation}
	and, if $\alpha$ is a $1$-form with respect to $[0,1]$,
		\begin{equation}
		\int_{0,\mathcal{L}}^{1}\alpha := (\int_{0}^{1} f(x,t)dt)\omega.
		\end{equation}
		This operator is very similar to the operator $p_\star$ (the integration along the fiber of $p: M \times [0,1] \longrightarrow M$), but they differ by sign. Consider $\alpha$ in $\Omega^k(M\times [0,1])$: we have $\int_{0\mathcal{L}}^1 \alpha = \pm p_\star\alpha$ where the choice between $+$ or $-$ depends on the degree of $\alpha$. This implies that the norm of $\int_{0\mathcal{L}}^1$ is equal to the norm of $p_\star$ as operators between the $\mathcal{L}^q$-spaces for each $q$ in $[1; + \infty)$. Then, by Proposition \ref{star}, we can see that $p_\star$ and $\int_{0,\mathcal{L}}^{1}$ are $\mathcal{L}^*$-bounded.
		\\We know from Lemma 11.4 of the book of Lee \cite{lee} that $\int_{0\mathcal{L}}^1: \Omega^*(M \times [0,1]) \longrightarrow \Omega^{*-1}(M)$ and that for all differential forms the equality (\ref{taglia}) is satisfied.
	\end{proof}
	\begin{prop}\label{K1}
		Let $id: (N,h) \longrightarrow (N,h)$ be the identity map on a manifold of bounded geometry and consider $p_{id}: T^\delta N \longrightarrow N$ the submersion related to the identity defined in Lemma \ref{tilde}. Then, if $pr_N: T^\delta N \longrightarrow N$ is the projection of the tangent bundle, then there is an $\mathcal{L}^*$-bounded operator $K_1: \mathcal{L}^*(N) \longrightarrow \mathcal{L}^*(T^\delta N)$ such that for every smooth form $\alpha$
		\begin{equation}
		p_{id}^*\alpha - pr_N^*\alpha = d \circ K_1\alpha + K_1\circ d\alpha.
		\end{equation}
		Moreover, if $\alpha$ is in $\Omega^*_c(N)$, then $K_1\alpha \in \Omega^*(T^\delta N)$ has compact support.
	\end{prop}
	\begin{proof}
		Observe that $pr_N$ is a Lipschitz submersion with bounded Fiber Volume and so $pr_N^*$ is a bounded operator.
		\\Moreover $p_{id}$ and the projection $pr_N: (T^\delta N, g_S) \longrightarrow (N, h)$ are homotopic. In particular, the homotopy $H(w_p, s) = p_{id}(s \cdot w_p)$ is a Lipschitz map. Moreover, by Lemma \ref{homo1} we know that $H$ is R.-N.-Lipschitz and so $H^*$ is an $\mathcal{L}^*$-bounded operator.
		\\This means that for all $\alpha$ in $\Omega^*_c(N)$, because of Lemma 11.4. of \cite{lee}
		\begin{equation}
		p_{id}^* - pr_N^* (\alpha)= i_1^*H^*\alpha - i_0^*H^*\alpha = \int_{0\mathcal{L}}^1 H^* d\alpha+ d \int_{0\mathcal{L}}^1 H^*\alpha.
		\end{equation}
		Then $K_1:= \int_{0\mathcal{L}}^1 \circ H^*$ satisfies (\ref{K1}) and it is $\mathcal{L}^*$-bounded.
	\end{proof}
 \begin{prop}\label{K2}
Consider $g: (M,m) \longrightarrow (N,h)$ and $f:(N,h) \longrightarrow (S,r)$ two smooth uniformly proper Lipschitz maps between manifolds of bounded geometry. Let $\delta \leq inj_N$ and $\sigma \leq inj_S$. Denote by $\overline{g}: g^*(T\delta N) \longrightarrow S$ the map $g \circ pr_M$, where $pr_M: g^*(T\delta N) \longrightarrow M$. 
\\Then there is an $\mathcal{L}^*$-bounded operator $K_2: \mathcal{L}^*(N) \longrightarrow \mathcal{L}^*((f \circ \overline{g})^*T^\delta S)$ such that for every smooth form $\alpha$
\begin{equation}
p_{f \circ p_g}^*\alpha - p_{f\circ \overline{g}}^*\alpha = d \circ K_2\alpha + K_2 \circ d\alpha.
\end{equation}
Moreover, if $\alpha$ is in $\Omega^*_c(S)$, then $K_2\alpha \in \Omega^*((f \circ \overline{g})^*T^\delta S)$ has compact support.
\end{prop}
	\begin{proof}
		Consider $h: g^*T^\delta N  \times [0,1] \longrightarrow S$ the map $h(w_{g(p)}, s) := f \circ p_g(s \cdot w_{g(p)})$. This is a Lipschitz homotopy between $f \circ p_g$ and $f \circ \overline{g}$. Let us define the submersion related to $h$ defined in Lemma \ref{tilde}
		\begin{equation}
		p_h: (f \circ \overline{g})^*(T^\sigma S) \times [0,1] \longrightarrow S.
		\end{equation}
		Actually the domain of $p_h$ should be $h^*T^\delta S$ and the domain of $f \circ p_g$ should be $(f \circ p_g)^*T^\delta S$, but we are considering these maps defined up to R.N.-lipschitz equivalence of vector bundle (see Remark \ref{equalizer}).
		\\Observe that $p_h$ is the homotopy between $p_f \circ P_g = p_{f \circ p_g}$ and $p_{f\circ \overline{g}}$ defined in Lemma \ref{homo2} and recall that $p_h$ is a R.-N.-Lipschitz map. Let us define the operator $K_2 := \int_{0\mathcal{L}}^1 \circ p_h^*$.
		It is an $\mathcal{L}^*$-bounded operator because it is a composition of $\mathcal{L}^*$-bounded operators. Then for every smooth form $\alpha$ we have
		\begin{equation}
		\begin{split}
		(p_{f \circ p_g})^* \alpha - p_{f\circ \overline{g}}^*\alpha &= (i_0^* - i_1^*)p_h^*\alpha \\
		&= (d \circ \int_{0\mathcal{L}}^1) p_h^* \alpha + ( \int_{0\mathcal{L}}^1 \circ d) p_h^*\alpha = d \circ K_2 \alpha + K_2 \circ d \alpha.
		\end{split}
		\end{equation}
		Finally, since $p_h$ is a proper map (it is a composition of proper maps), if $\alpha \in \Omega^*_c(S)$ the support of $K_2\alpha \in \Omega^*((f \circ \overline{g})^* T^\delta S)$ is compact.
	\end{proof}
	\begin{prop}\label{K3}
		Let $f_0$ and $f_1:(M,m) \longrightarrow (N,l)$ be two smooth, uniformly proper and Lipschitz maps between manifolds of bounded geometry. Let us suppose that $f_1 \sim_\Gamma f_0$ with a smooth Lipschitz homotopy. Then there is an $\mathcal{L}^*$-bounded operator $K_3: \mathcal{L}^*(N) \longrightarrow \mathcal{L}^*(p_{f_1}^*T^\delta N)$ such that for all smooth form $\alpha$ 
		\begin{equation}
		p_{f_1}^* \alpha - p_{f_0}^* \alpha = d \circ K_3 \alpha + K_3\circ d \alpha.
		\end{equation}
		Moreover if $\alpha \in \Omega^*_c(N)$ then the support of $K_3\alpha \in \Omega^*(p_{f_1}^*T^\delta N)$ is compact.	
	\end{prop}
	\begin{proof}
		Observe that $h$, the homotopy such that $h(p, 0) = f_1(p)$ and $h(p,1) = f_0(p)$, is a uniformly proper Lipschitz map. Let us define the metric $g_S$ on $f_1^*T^\delta N$ as the Sasaki metric defined by using the Riemannian metric $m$, the bundle metric $f_1^*l$ and the connection $f^*\nabla^{LC, N}$.
		\\By applying Corollary \ref{R.-N.-p}, we obtain that $p_h: (f_1^*T^\delta N \times  [0,1], g_S + dt^2) \longrightarrow (N,h)$ is a R.-N.-Lipschitz map and so $p_h^*$ is an $\mathcal{L}^*$-bounded operator.
		\\Moreover $p_h$ is a Lipschitz-homotopy between $p_{f_1}$ and $p_{f_0}$. This fact follows directly by the definition of submersion related to a Lipschitz map in Lemma \ref{tilde} and by Remark \ref{equalizer}.
		\\So we can conclude as well as we did in the Proposition \ref{K1} and Proposition \ref{K2} by considering $K_3 := \int_{0\mathcal{L}}^1 \circ p_h^*$ and by using that $p_h$ is a proper map.
	\end{proof}
	\begin{prop}\label{drgfin}
		Consider $(M,g)$,$(N,h)$ and $(S,l)$ three manifolds of bounded geometry and consider $f, \tilde{f}:(M,g) \longrightarrow (N,h)$, $F:(M,g) \longrightarrow (N,h)$ and $g:(S,l)  \longrightarrow (M,g)$ three uniform maps, possibly non-smooth. Then, in $L^{q,p}$-cohomology, 
		\begin{enumerate}
			\item $T_{id_M} = Id_{H^*_{q,p}(M)}$,
			\item let us consider a Thom form $\omega$ on $\Omega^*(TN)$. Assume that the pointwise norm of $\omega$ is uniformly bounded and that the support of $\omega$ is strictly contained in $T^{inj_N}N$\footnote{For example as the Thom form defined in Subsection \ref{forma}.}. Let us suppose that $f \sim_\Gamma \tilde{f}$ with a smooth Lipschitz homotopy. Then, in $L^{q,p}$-cohomology,
			\begin{equation}
			T_{f} =  pr_{M\star} \circ e_{\tilde{F}^*\omega} \circ p_f^*,
			\end{equation}
			\item if $f$ is not differentiable and $f'$ is a smooth Lipschitz maps which is Lipschitz-homotopic to $f$, then $T_{f}$ does not depend on the choice of $f'$,
			\item if $f \sim_\Gamma F$ then $T_f = T_F$,
			\item $T_{f \circ g} = T_g \circ T_f$,
			\item if $f$ is an $L^*$-map\footnote{This means that $f$ is a smooth R.-N.-Lipschitz map which is also uniformly proper} then $f^* = T_f$,
		\end{enumerate}
		Moreover the identities above also holds in reduced $L^{q,p}$-cohomology and in $L^{q,p}$-quotient cohomology.
	\end{prop}
	\begin{proof}
			\textbf{Point 1.} Let us consider the standard projection $pr_M: TM \longrightarrow M$ and let $\delta \leq inj_M$. Because of Lemma \ref{K1}, for all smooth forms $\alpha$ in $\Omega^*_c(M)$,
		\begin{equation}
			pr_M^* - p_{id}^*(\alpha) = d \circ K_1 + K_1 \circ d (\alpha).
		\end{equation}
		Then the identity map in $\mathcal{L}^p(M)$ can be written as
		\begin{equation}
			1(\alpha) := pr_{M \star} \circ e_{\omega} \circ pr^*_M(\alpha)
		\end{equation}
		where $pr_{M \star}$ is the operator of integration along the fibers of $pr_M$ and $\omega$ is a Thom form with uniformly bounded pointwise norm and support contained in $T^\delta M$ (see subsection \ref{forma}).
		\\Then for every $\alpha \in \Omega_c^*(M)$
		\begin{equation}
			\begin{split}
				1 - T_{id_M}(\alpha) &= pr_{M \star} \circ e_{\omega} \circ (pr^*_M - p_{id}^*)\alpha\\
				&= pr_{M \star} \circ e_{\omega} \circ (d \circ K_1 + K_1 \circ d) \alpha.
			\end{split}
		\end{equation}
		Observe that since $\omega$ is closed,  $d(\alpha \wedge \omega) = (d\alpha) \wedge \omega$. Moreover $\alpha \wedge \omega$ is in $\Omega^*_{vc}(TM)$. This means that the exterior derivative can be switched with $pr_{M \star}$ and so
		\begin{equation}
			\begin{split}
				1 - T_{id_M}(\alpha) &= d \circ  pr_{M \star} \circ e_{\omega} \circ K_1 + pr_{M \star} \circ e_{\omega} \circ K_1 \circ d (\alpha)\\
				&=d \circ Y_1 + Y_1 \circ d(\alpha) \label{Y_1},
			\end{split}
		\end{equation}
		where $Y_1 := pr_{M \star} \circ e_{\omega} \circ K_1$. Observe that $Y_1$ is an $\mathcal{L}^*$-bounded operator because it is a composition of $\mathcal{L}^*$-bounded operators. By applying Proposition \ref{boundedness}, we obtain that $Y_1(dom(d)) \subseteq dom(d)$ and so (\ref{Y_1}) holds for every $\alpha$ in $dom(d_{min})$. Then in $L^{q,p}$-cohomology we obtain
		\begin{equation}
			Id_{H^*_{q,p}(M)} = T_{id_M}.
		\end{equation}
		\\\textbf{Point 2.} Consider $\omega$ a Thom form of $TN$ which has support contained in $T^\delta N$ and uniformly bounded pointwise norm (for example the Thom form of Subsection \ref{forma}) and let $h: M \times [0,1] \longrightarrow N$ be a smooth Lipschitz homotopy between $f$ and $\tilde{f}$. Consider $g_S$ the Sasaki metric on $T^\delta N$ induced by the metric on $N$ and by the Levi-Civita connection on $TN$. Moreover let us denote by $g_{S,f}$ the Sasaki metric on $f^*TN$ induced by the metrics on $M$ and on $N$ and the pullback of the Levi-Civita connection on $TN$. Then the fibered map induced by $h$
		\begin{equation}
		\begin{split}
 H: (h^*T^\delta N, g_{Sh}) &\longrightarrow (T^\delta N, g_S) \\
 (p,t, w_{h(p,t)}) &\longrightarrow  w_{h(p,t)}
		\end{split}
		\end{equation}
		is Lipschitz. By Remark \ref{equalizer} we can see $H$ as a Lipschitz map defined on $(f^*(T^\delta N) \times [0,1], g_{S,f} + dt^2)$.  Observe that $F^*\omega = i_0^*H^*\omega$ and $\tilde{F}^*\omega = i_1^*H^*\omega$ where $i_j: M \longrightarrow M \times [0,1]$ is defined by setting $i_j(p) = (p,j)$ for each $j$ in $[0,1]$.
		\\Observe that
		\begin{equation}
		 F^*\omega - \tilde{F}^*\omega = d \int_{0\mathcal{L}}^1 H^*\omega - \int_{0\mathcal{L}}^1 H^* d \omega = d \int_{0\mathcal{L}}^1 H^*\omega.
		\end{equation}
		Then, if we consider $\alpha$ an $\mathcal{L}^p$ smooth form on $f^*T^\delta N$, we obtain
		\begin{equation}
		  \begin{split}
		  e_{F^*\omega} - e_{\tilde{F}^*\omega} (\alpha) &= \alpha \wedge d \int_{0\mathcal{L}}^1 H^*\omega \\
		  &= (-1)^{deg(\alpha)} d(\alpha \wedge \int_{0\mathcal{L}}^1 H^*\omega) - (-1)^{deg(\alpha)} d\alpha \wedge \int_{0\mathcal{L}}^1 H^*\omega\\
		  &= (-1)^{deg(\alpha)} d(\alpha \wedge \int_{0\mathcal{L}}^1 H^*\omega) - (-1)^{deg(\alpha)} \int_{0\mathcal{L}}^1  d\alpha \wedge H^*\omega \\
		  &= d \circ K - K \circ d (\alpha)
		  \end{split}
		  \end{equation}
		  where $K (\alpha) = (-1)^{deg(\alpha)} \circ \int_{0\mathcal{L}}^1 \circ e_{H^*\omega}(\alpha).$
		Observe that the pointwise norm of $\omega$ is uniformly bounded in $TN$: it follows from subsection \ref{forma}. Then, since $H$ is a Lipschitz map, also the norm of $H^*\omega$ is uniformly bounded in $f^*T^\delta N \times [0,1]$. This means that $e_{H^*\omega}$ is an $\mathcal{L}^*$-bounded operator.
		\\This implies that for each smooth form $\alpha$ in $\mathcal{L}^p(N)$ for all $p \in [1, +\infty)$,
		\begin{equation}
		\begin{split}
		&T_f - pr_{M\star} \circ e_{\tilde{F}^*\omega} \circ p_f^*(\alpha) = pr_{M\star} \circ ( e_{F^*\omega} - e_{\tilde{F}^*\omega}) \circ p_f^*\alpha \\
		&= pr_{M \star} \circ (d \circ K - K \circ d)  \circ p_f^*\alpha = d \circ J - J d (\alpha),
		\end{split}
		\end{equation}
		where $ J :=  pr_{M \star} \circ K \circ p_f^*$.
		Observe that $J$ is $\mathcal{L}^*$-bounded because it is a composition of $\mathcal{L}^*$-bounded operators. Then, since $J(\Omega_c^*(N)) \subseteq \Omega_c^*(M)$, we conclude the proof of Point 2 by applying Proposition \ref{boundedness}.
		\\
		\textbf{Point 3.} Consider two smooth Lipschitz approximations $f_1$ and $f_2$ of $f$. This in particular means that $f_1 \sim_\Gamma f_2$. 
Let us consider, now, a smooth Lipschitz homotopy $H$ between $f_1$ and $f_2$ as in Lemma \ref{C^k_bhomo}. We know, thanks to Proposition \ref{K3}, that there is an $\mathcal{L}^*$-bounded operator $K_3$, such that for every $\alpha$ in $\Omega_c^*(N)$
\begin{equation}
	p_{f_1}^* - p_{f_2}^* (\alpha)= d \circ K_3 + K_3 \circ d (\alpha).
\end{equation}
Moreover, by the previous point, if we consider the Thom form $F_1^* \omega$ of $f_1^*TN$ defined as in subsection \ref{forma}, we obtain that
\begin{equation}
T_{f_2} = pr_{M\star} \circ e_{F_1^* \omega} \circ p_{f_2}^*.
\end{equation}
Moreover, by the previous point, we can also suppose that
\begin{equation}
T_{f_1} = pr_{M\star} \circ e_{F_1^* \omega} \circ p_{f_1}^*.
\end{equation}
This means that if we consider the $\mathcal{L}^*$-bounded operator\footnote{It is $\mathcal{L}^*$-bounded because it is a composition of $\mathcal{L}^*$-bounded operators} $Y_3 := pr_{M\star} \circ e_{F_1^* \omega} \circ K_3$, then for all $\alpha$ in $\Omega_c^*(N)$ we obtain
\begin{equation}
	\begin{split}
		&T_{f_1} - T_{f_2} (\alpha) = pr_{M\star} \circ e_{F_1^* \omega} \circ (p_{f_1}^* - p_{f_2}^*) \\
		&= pr_{M\star} \circ e_{F_1^* \omega} \circ (d \circ K_3 + K_3 \circ d) (\alpha) = d \circ Y_3 + Y_3 \circ d (\alpha).
	\end{split}
\end{equation}
Again, by applying Proposition \ref{boundedness}, we obtain that the identity above holds for all $\beta$ in $dom(d_{min})$.
\\Then, in $L^{q,p}$-cohomology, $T_f$ does not depend on the choice of the smooth Lipschitz approximation of $f$.
\\\textbf{Point 4.} Let us consider $f'$ and $F'$ two differentiable maps which are Lipschitz-homotopic to $f$ and $F$. Let $H$ be a smooth Lipschitz homotopy between $f'$ and $F'$. Then we can conclude as the previous point.		
\\
	\textbf{Point 5.} Thanks to the previous points, we can consider $f$ and $g$ as smooth maps. Let us define the vector bundles $pr_S: g^*T^\sigma M \longrightarrow S$, $pr_{S2}: (f \circ g)^*T^\delta N \longrightarrow S$. Moreover denote by $\overline{g}: g^*T^\sigma M \longrightarrow M$ the map defined as $\overline{g} := g \circ pr_S$. We know that
	\begin{equation}
	(f \circ \overline{g})^*TN = pr_{S2}^* g^* (TM) = pr_S^* (f \circ g)^*TN =  (f \circ g)^*TN \oplus g^*TM
	\end{equation}
	as vector bundles over $S$. Then, we obtain the following diagram
	\begin{equation}
	    \xymatrix{
& (f \circ \overline{g})^*T^\delta N \ar[ld]^{PR_S} \ar[d]^{PR_{S2}} \ar[rd]^{P_g} & &\\
(f \circ g)^*T^\delta N \ar[rd]^{pr_{S2}} &g^*T^\sigma M \ar[d]^{pr_S} \ar[rd]^{p_g} & f^*T^\delta N \ar[d]^{pr_M} \ar[rd]^{p_f} &\\
&S \ar[r]^{g} & M \ar[r]^{f} & N}
	\end{equation}
	where $PR_S$, $PR_{S2}$ and $P_g$ are the bundle maps induced by $pr_S$, $pr_{S2}$ and by $p_g$. In particular
	\begin{equation}\label{laggiu}
	 pr_{S2} \circ PR_S = pr_S \circ PR_{S2}.
	\end{equation}
	Observe that on $(f \circ \overline{g})^*TN = pr_{S2}^* g^* (TM)$ there is the Sasaki metric $g_1$ induced by the Sasaki metric on $(f \circ g)^*TN$, the pullback bundle metric of $g^*(TM)$ induced by the metric on $M$ and the pullback connection of $\nabla^{LC}_M$.
	\\Moreover if we consider $(f \circ \overline{g})^*TN = pr_S^* (f \circ g)^*TN$, then we can define the Sasaki metric $g_2$ induced by the Sasaki metric on $g^* (TM)$, the pullback bundle metric of $(f \circ g)^*TN$ induced by the metric on $N$ and the pullback connection of $\nabla^{LC}_N$.
	\\In order to prove this equality between metrics, let us fix some normal coordinates $\{x^i\}$ around a point $s$ in $S$, $\{y^j\}$ around $g(s)$ and $\{z^k\}$ around $f \circ g (s)$. Then we consider on $(f \circ \overline{g})^*TN$ the fibered coordinates $\{x^i, \mu^j, \sigma^k\}$ related to the coordinates $\{x^i\}$, the frame $\{\frac{\partial}{\partial y^j}\}$ and the frame $\{\frac{\partial}{\partial z^k}\}$. Then by the formula \ref{metri}, we obtain that the matrices related to $g_1$ and $g_2$ on a point $(0, \mu, \sigma)$ are both the identity. So $g_1 = g_2$.
	\\This means that, with respect to the metric $g_1$, the maps $PR_S$, $PR_{S_2}$ are R.N.-Lipschitz. Moreover, since $P_g \sim \overline g$, we also have that $P_g$ is R.N.-Lipschitz. Finally, because of Proposition \ref{star}, we also have that $PR_{S\star}$ and $PR_{S2\star}$ are both $\mathcal{L}^*$-bounded oprators.
	\\Consider the submersions $p_f: f^*T^\delta N \longrightarrow N$, $p_g: g^*T^\sigma M \longrightarrow M$, $p_{f \circ g}: (f \circ g)^*T^\delta N \longrightarrow N$ related to $f$, $g$, and $f \circ g$. Then, as a consequence of the previous Proposition, we obtain for every $\alpha$ in $\Omega^*_c(N)$ 
		\begin{equation}
		(P_g)^* \circ p_f^* - \overline{p}^*_{f \circ g}(\alpha)= d \circ K_2 + K_2 \circ d (\alpha). 
		\end{equation}
		Let us denote by $\omega$ and $\omega'$ some Thom forms of $TN$ and of $TM$ with uniformly bounded pointwise norms and support contained in $T^\delta N$ and $T^\sigma M$.
		\\Let us denote by $FG: (f\circ g)^*T^\delta N \longrightarrow T^\delta N$ the bundle map induced by $f \circ g$. Moreover we also denote by $F:f^*T^\delta N \longrightarrow T^\delta N$ and $G:g^*T^\sigma M \longrightarrow T^\sigma M$ the bundle maps induced by $f$ and $g$.
		 \\Then 
		\begin{equation}
		\begin{split}
		T_{f \circ g} &= pr_{S2 \star} \circ e_{(FG)^*\omega} \circ p_{f\circ g}^* \\
		&=pr_{S2 \star} \circ id_{(f \circ g)^* T^\delta N} \circ e_{(FG)^*\omega} \circ p_{f\circ g}^* \\
		&=pr_{S2 \star} \circ PR_{S, \star} \circ e_{PR_{S2}^* G^* \omega'} \circ PR_S^*  \circ e_{(FG)^*\omega} \circ p_{f\circ g}^*\\
		&=pr_{S2 \star} \circ PR_{S, \star} \circ e_{PR_{S2}^* G^* \omega'} \circ e_{PR_S^* (FG)^*\omega}  \circ PR_S^* \circ p_{f\circ g}^*,
		\end{split}
		\end{equation}
		Observe that $p_{f \circ \overline{g}} = p_{f\circ g} \circ PR_S$, and so 
		\begin{equation}
		T_{f \circ g} = pr_{S2 \star} \circ PR_{S, \star} \circ e_{PR_{S2}^* G^* \omega'} \circ e_{PR_S^* (FG)^*\omega}  \circ p_{f \circ \overline{g}}^*. 
		\end{equation}
		Now we will focus on $T_g \circ T_f$. Observe that
		\begin{equation}
		T_g \circ T_f = pr_{S \star} \circ e_{G^* \omega'} \circ p_g^* \circ pr_{M \star} \circ e_{F^*\omega} \circ p_f^*.
		\end{equation}
		It is possible to apply the Proposition VIII of Chapter 5 in \cite{Conn} to the fiber bundles $((f \circ \overline{g})^*T^\delta N , PR_{S2}, g^*T^\sigma M, B^\delta)$ and $(T^\sigma M, pr_{M}, M, B^\sigma)$ and the bundle morphism $P_g$ induced by $p_g$. We obtain that $p_g^* \circ pr_{M \star} = PR_{S2,\star} \circ P_g^*$.
	Moreover $P_g^* \circ e_{F^*\omega} = e_{P_g^* F^*\omega} \circ P_g^*$.
		This means
		\begin{equation}
		\begin{split}
		T_g \circ T_f &= pr_{S \star} \circ e_{G^*\omega'} \circ p_g^* \circ pr_{M \star} \circ e_{F^*\omega} \circ p_f^* \\
		&= pr_{S \star} \circ e_{G^*\omega'} \circ PR_{S2,\star} \circ e_{P_g^* F^*\omega} \circ P_g^*\circ p_f^*\\
		&= pr_{S \star} \circ PR_{S2,\star} \circ e_{PR_{S2}^* G^*\omega'} \circ e_{P_g^* F^*\omega} \circ P_g^*\circ p_f^*
		\end{split}
		\end{equation}
		Because of (\ref{laggiu}) we obtain $pr_{S \star} \circ PR_{S2,\star} = pr_{S2 \star} \circ PR_{S,\star}.$
		\\Moreover  $PR_S^* (FG)^*\omega$ is the pullback of $\omega$ along the bundle map induced by $f \circ g \circ pr_{S2}$, but it can be also seen as the pullback of $\omega$ along the bundle map induced by $f \circ g \circ pr_{S} = f \circ \overline{g}$.
		\\On the other hand $P_g^* F^*\omega$ is the pullback of $\omega$ along the bundle map induced by  $f \circ p_g$. Then, by Point 2 and since $f \circ \overline{g}$ and $f \circ p_g$ are uniformly homotopy equivalent, there is an $\mathcal{L}^*$-bounded operator $K$ such that
		\begin{equation}
		 e_{PR_S^* (FG)^*\omega} - e_{P_g^* F^*\omega}(\alpha) = d \circ K - K \circ d (\alpha)
		\end{equation}
		and $K\beta$ has compact support if $\beta$ has compact support.
		\\This means that, on $\Omega_c^*(N)$, because of Proposition \ref{K2}, we have
		\begin{equation}\label{piom}
		\begin{split}
		T_{f \circ g} - T_g \circ T_f &= pr_{S2 \star} \circ PR_{S, \star} \circ e_{PR_{S2}^* G^* \omega'} \circ [dK - K d]  \circ p_{f \circ \overline{g}}^* \\
		&+pr_{S2 \star} \circ PR_{S, \star} \circ e_{PR_{S2}^* \circ G^* \omega'} \circ e_{PR_S^* (FG)^*\omega} \circ [K_2 \circ d + d \circ K_2] \\
		&= d \circ Q - Q \circ d +d \circ W  + W \circ d
		\end{split}
		\end{equation}
		where $Q := pr_{S2 \star} \circ PR_{S, \star} \circ e_{PR_{S2}^* G^* \omega'} \circ K \circ p_{f \circ \overline{g}}^*$ and $W := pr_{S2 \star} \circ PR_{S, \star} \circ e_{PR_{S2}^* G^* \omega'} \circ e_{PR_S^* (FG)^*\omega} \circ K_2.$
		Observe that $Q$ and $W$ are $\mathcal{L}^*$-bounded operators. Moreover $W(\Omega_c^*(N)) \subseteq \Omega^*_c(S)$ and $Q(\Omega_c^*(N)) \subseteq \Omega^*_c(S)$. Because of Proposition \ref{boundedness}, $W(dom(d_{p,q})) \subseteq dom(d_{p,q})$ and $Q(dom(d_{p,q})) \subseteq dom(d_{p,q})$ and so, in $L^{q,p}$-cohomology,
		\begin{equation}
		T_{f \circ g} = T_g \circ T_f.
		\end{equation}
		\\\textbf{Point 6.} In order to prove this statement we have to observe that, by Lemma \ref{tilde}, $p_f = p_{id} \circ F$. Let us consider a form $\alpha \in \Omega_c^*(N)$:
		\begin{equation}
		T_f \alpha = pr_{M\star} \circ e_{\omega} \circ F^* \circ p_{id}^* \alpha = pr_{M\star} \circ F^* \circ e_{\omega} \circ p_{id}^* \alpha.
		\end{equation}
		Now, thanks to the Proposition VIII of Chapter 5 of \cite{Conn},
		\begin{equation}
\begin{split}
		T_f \alpha &= pr_{M\star} \circ F^* \circ e_{\omega} \circ p_{id}^* \alpha \\
		&= f^* \circ pr_{N\star} \circ e_{\omega} \circ p_{id}^* \alpha = f^* \circ T_{id_N} \alpha \label{impl}
		\end{split}
		\end{equation}
		Since $f$ is R.-N.-Lipschitz, $f^*$ is $\mathcal{L}^*$-bounded. Then (\ref{impl}) implies $T_f = f^* \circ T_{id_N}$.
		So on $dom(d_{min})$ the following holds
		\begin{equation}
		f^* - T_f = f^* \circ (1 - T_{id_N}) = f^* \circ (d \circ Y_1 + Y_1 \circ d).
		\end{equation}
		Observe that $f^*(\Omega_c^*(N)) \subseteq \Omega_c^*(M)$ since $f$ is proper. Moreover $f^*d\alpha= d f^* \alpha$ for all smooth form $\alpha$. Then on $dom(d_{min})$
		\begin{equation}
		f^* - T_f = f^* \circ (d \circ Y_1 + Y_1 \circ d) = d \circ W + W \circ d		
		\end{equation}
		where $W = f^* \circ Y_1.$ And so in $L^{q,p}$-cohomology
		\begin{equation}
		f^* = T_f.
		\end{equation}
		In order to conclude the proof we have to show that all the identities above also hold in reduced $L^{q,p}$-cohomology and in $L^{q,p}$-quotient cohomology. To this end observe that if $Q = d Z \pm Zd$ as operator on $dom(d)$, and $Z$ is an $\mathcal{L}^*$-bounded operator then, in reduced $L^{q,p}$ cohomology, $Q$ is the null operator. This is a consequence of Corollary \ref{boundedness2}.
	\end{proof}
	\begin{cor}\label{finale}
		Let $(M,g)$ and $(N,h)$ be two oriented manifolds of bounded geometry. Let $f:(M,g) \longrightarrow (N,h)$ be a uniform homotopy equivalence. Then the (reduced or not) $L^{q,p}$-cohomologies and the $L^{q,p}$-quotient cohomology groups are isomorphic.
	\end{cor}
	\begin{proof}
	Since $f$ and its homotopy inverse $g$ are uniform homotopy equivalences, then they are uniform map. Moreover also $g \circ f$ and $f \circ g$ are uniformly homotopic to the identities. Then
		\begin{equation}
		1_{H^*_{q,p}(M)} = T_{id_M} = T_{g \circ f} = T_f \circ T_g \mbox{   and   } 1_{H^*_{q,p}(N)} = T_{id_N} = T_{f \circ g} = T_g \circ T_f.
		\end{equation}
		The reduced case and the quotient cohomology case can be proved exactly with the same argument. 
	\end{proof}
	\subsection{Consequence: uniform homotopy invariance of $L^2$-index of signature operator}
	A consequence of this work is the uniform homotopy invariance, for manifolds of bounded geometry, of the index of the $L^2$-version of the signature operator $d + d^*$ defined by Bei in page 20 of \cite{Bei3}. Fix $(M,g)$ an oriented manifold of bounded geometry and let $dim(M) =4k$ for some positive $k$ in $\mathbb{N}$. Let us consider
	\begin{equation}\label{pairing}
	\begin{split}
	  \langle \cdot , \cdot :\rangle_M  &\overline{H}^i_{2}(M, \mathbb{R}) \times \overline{H}^i_{2}(M, \mathbb{R}) \longrightarrow \mathbb{R} \\
	    &([\eta], [\omega]) \longrightarrow \int_{M} \eta \wedge \omega
	    \end{split}
	\end{equation}
	where $\overline{H}^i_{2}(M, \mathbb{R})$ denotes the $i$-th group of reduced $L^2$-cohomology defined by using forms with value in $\mathbb{R}$. This is a well-defined and non-degenerate pairing: the proof is exactly the same given by Bei in Proposition 4.1 of \cite{Bei3}. In particular, if $i = 2k$ we obtain a symmetric bilinear form. Then, if $\overline{H}^{2n}_{2}(M, \mathbb{R})$ is finite dimensional, we can denote by $\sigma_M$ the signature of the pairing.
	\\Let us consider the signature operator $d_M + d^*_M$ where $d_M := d_{min,M} = d_{max,M}$ and let $d^*_M$ be the adjoint of $d_M$. In page 20 of \cite{Bei3}, the \textit{index of the signature operator} $ind((d_M + d_M^*)^+)$  is defined and, in Theorem 4.2, the author proves
	\begin{equation}\label{sign}
	    \sigma_M = ind((d_M + d_M^*)^+).
	\end{equation}
	\begin{prop}
	Let $(M,g)$ and $(N,h)$ be two oriented manifolds of bounded geometry and let $dim(M) = dim(N) = 4k$ for some positive $k$ in $\mathbb{N}$. Let  $f: (M,g) \longrightarrow (N,h)$ be a uniform homotopy equivalence which preserves the orientations and assume that $\overline{H}^{2k}_{2}(M, \mathbb{R}) \cong \overline{H}^{2k}_{2}(N, \mathbb{R})$ are finite-dimensional. Then
	\begin{equation}
	    ind((d_M + d_M^*)^+) = ind((d_N + d_N^*)^+)
	\end{equation}
	\end{prop}
	\begin{proof}
	In order to prove the statement it is sufficient to prove that $T_f$ induces an isometry between $\overline{H}^{2n}_{2}(M, \mathbb{R})$ and $\overline{H}^{2n}_{2}(N, \mathbb{R})$ which respects the bilinear form defined in (\ref{pairing}).
	\\Let us consider $[\eta]$ and $[\omega]$ two classes in $\overline{H}^{2n}_{2}(N, \mathbb{R})$. Thanks to Theorem 12.7 of \cite{Gol}, we can suppose that $\eta$ and $\omega$ are smooth forms. Observe that, since the operator $Y_1$ in point 1 of Proposition \ref{drgfin} is $\mathcal{L}^*$-bounded,  $T_{id} \eta = \eta + d\alpha$ and $T_{id} \omega = \omega + d\beta$ where $\alpha$ and $\beta$ are smooth forms. We obtain that
	\begin{equation}
	    \begin{split}
	   \langle [T_f \eta] , [T_f \omega] \rangle_M &= \int_M T_f\eta \wedge T_f\omega \\
	   &= \int_M f^*(T_{id}\eta) \wedge f^* (T_{id}\omega) = \int_M f^* (T_{id}\eta \wedge T_{id}\omega) \\
	   &= deg(f) \cdot \int_N  T_{id}\eta \wedge T_{id}\omega = \int_N (\eta + d\alpha) \wedge (\omega + d\beta)\\
	   &= \int_N \eta \wedge \omega + \int_N d(\alpha \wedge \omega) + \int_N d(\eta \wedge \beta) + \int_N d(\alpha \wedge d\beta)\\
	   &=\int_N \eta \wedge \omega + 0 = \langle [\eta] , [\omega]\rangle_N.
	    \end{split}
	\end{equation}
	\end{proof}

\end{document}